\documentclass{amsart}
\usepackage[normalem]{ulem}
\usepackage{amssymb,epsfig,mathrsfs,mathpazo,stackengine,scalerel,url,amsthm,thmtools,bbm}
\usepackage{xcolor}
\usepackage{stmaryrd}
\usepackage{epsfig,bm}
\usepackage[multiple]{footmisc}
\usepackage{accents} % ensures that ${\bar{Y}^\delta}'$ doesn't give a double subscript error. 
\usepackage{mathtools} % for \mathclap
\usepackage{footmisc}

\makeatletter
\newcommand*\bigcdot{\mathpalette\bigcdot@{.5}}
\newcommand*\bigcdot@[2]{\mathbin{\vcenter{\hbox{\scalebox{#2}{$\m@th#1\bullet$}}}}}
\makeatother

\newcommand{\osc}{{\rm osc}}

\usepackage{arydshln, booktabs, makecell, pbox, tabularx}

\usepackage{graphicx, caption, subcaption}

\newtheorem{theorem}{Theorem}[section]
\newtheorem{lemma}[theorem]{Lemma}

\newtheorem{corollary}[theorem]{Corollary}

\declaretheorem[style=remark,qed=$\vartriangle$,sibling=theorem]{remark}
\numberwithin{equation}{section}

\newcommand{\eps}{\varepsilon}

\newcommand{\R}{\mathbb R}

\newcommand{\N}{\mathbb N}

\newcommand{\cN}{\mathcal N}
\newcommand{\cD}{\mathcal D}

\newcommand{\cH}{\mathcal H}
\newcommand{\cP}{\mathcal P}
\newcommand{\cE}{\mathcal E}

\newcommand{\cL}{\mathcal L}

\newcommand{\identity}{\mathrm{Id}}

\DeclareMathOperator{\ran}{ran}
\DeclareMathOperator{\supp}{supp}

\DeclareMathOperator{\diam}{diam}
\DeclareMathOperator*{\argmin}{argmin}
\DeclareMathOperator{\dist}{dist}

\DeclareMathOperator{\Span}{span}

\newcommand{\nrm}{| \! | \! |}

\newcommand{\be}{\begin{equation}}
\newcommand{\ee}{\end{equation}}

\newcommand{\bbT}{\mathbb{T}}
\newcommand{\tria}{{\mathcal T}}

\newcommand{\uumlaut}{{\"u}}

{\par\begin{samepage}%
\begin{tabbing}\ttfamily%
 \hspace*{5mm}\=\hspace{3ex}\=\hspace{3ex}\=\hspace{3ex}\=\hspace{3ex}%
\=\hspace{3ex}\=\hspace{3ex}\=\hspace{3ex}\=\hspace{3ex}\kill}%
{\end{tabbing}\end{samepage}}
\newcounter{ccondition}

{%
\renewcommand{\theequation}{\temp}%
\addtocounter{equation}{-1}%
\par\noindent%
\ignorespacesafterend%
}

\newcounter{mylistcounter}
\renewcommand{\themylistcounter}{(\roman{mylistcounter})}

\makeatletter
\@namedef{subjclassname@2020}{%
  \textup{2020} Mathematics Subject Classification}
\makeatother

\title[Ultra-weak least squares discretizations ]{Ultra-weak least squares discretizations for Unique Continuation and Cauchy problems}
\date{\today}

\author{Harald Monsuur, Rob Stevenson}
\address{Korteweg-de Vries (KdV) Institute for Mathematics, University of Amsterdam, P.O. Box 94248, 1090 GE Amsterdam, The Netherlands.}
\email{h.monsuur@uva.nl, rob.p.stevenson@gmail.com}

\thanks{This research has been supported by the NSF Grant DMS ID 1720297, and by the Netherlands Organization for Scientific Research (NWO) under contract.~no.~SH-208-11. 
We acknowledge the support of SURF (www.surf.nl) in using the National Supercomputer Snellius.}



\begin{document}

\begin{abstract} In this paper, conditional stability estimates are derived for unique continuation and Cauchy problems associated to the Poisson equation in ultra-weak variational form. Numerical approximations are obtained as minima of regularized least squares functionals. The arising dual norms are replaced by discretized dual norms, which leads to a mixed formulation in terms of trial- and test-spaces.
For stable pairs of such spaces, and a proper choice of the regularization parameter, the $L_2$-error on a subdomain in the obtained numerical approximation can be bounded by the best possible fractional power of the sum of the data error and the error of best approximation. Compared to the use of a standard variational formulation, the latter two errors are measured in weaker norms. To avoid the use of $C^1$-finite element test spaces, nonconforming finite element test spaces can be applied as well. They either lead to the qualitatively same error bound, or in a simplified version, to such an error bound modulo an additional data oscillation term. Numerical results illustrate our theoretical findings.
\end{abstract}
\maketitle

%%%%%%%%%%%%%%
\section{Introduction}
%%%%%%%%%%%%%%
\subsection{UC and Cauchy problems}
In this work we consider the following two problems: The Unique Continuation (UC) problem that for bounded domains $\emptyset \neq \omega \Subset \Omega \subset \R^n$ seeks $u \colon \Omega \rightarrow \R$ such that
$$
\left\{\begin{array}{rl}
-\triangle u = \ell& \text{on }\Omega,\\
u|_{\omega}=q & \text{on }\omega;
\end{array}
\right.
$$
and the Cauchy problem that for some open connected Lipschitz subset $\emptyset \neq\Sigma \subset \partial\Omega$  seeks $u \colon \Omega \rightarrow \R$ such that 
\be \label{eq:41}
\left\{\begin{array}{rl}
-\triangle u = \ell& \text{on }\Omega,\\
u=g& \text{on }\Sigma,\\
\partial_n u=\psi& \text{on }\Sigma.
\end{array}
\right.
\ee
These Poisson problems are known to be ill-posed, but so-called \emph{conditional stability} can be demonstrated. 

To see what the latter means, let us consider the UC problem. The common variational formulation reads as
finding $u \in H^1(\Omega)$ such that
$$
\left\{\begin{array}{r@{}c@{}ll}
\int_\Omega \nabla u \cdot \nabla v\,dx &\,\,\,=\,\,\,& \ell(v)& (v \in H^1_0(\Omega)),\\
u|_{\omega}&\,\,\,=\,\,\,&q & \text{on }\omega.
\end{array}
\right.
$$
For $A \in \cL(H^1(\Omega), H^{-1}(\Omega) \times L_2(\omega))$ defined by $
 Az :=\big(v \mapsto \int_\Omega \nabla z\cdot \nabla v\,dx,z|_{\omega}\big)$, it is known, cf.~e.g.~\cite[Sect.~2.2]{35.9298}, 
 that for a domain $G$ with $\omega \Subset G \Subset \Omega$, there exists an $\alpha \in (0,1)$ such that  \emph{under the condition that $\|z\|_{L_2(\Omega)}\lesssim 1$}, it holds that
$$
 \|z\|_{L_2(G)} \lesssim \|A z\|_{H^{-1}(\Omega) \times L_2(\omega)}^\alpha \quad (\|A z\|_{H^{-1}(\Omega) \times L_2(\omega)} \rightarrow 0).
 $$

For completeness, here and in the following, with $C \lesssim D$ it is meant that $C$ can be bounded by a multiple of $D$, independently of parameters which $C$ and $D$ may depend on. Clearly, $C \gtrsim D$ means $D \lesssim C$, and $C \eqsim D$ means $C \lesssim D$ and  $C \gtrsim D$.

Conditional stability estimates have been the basis for the design of convergent numerical approximation schemes (e.g.~\cite{35.8585,35.925,35.859,35.8595,35.9298} and \cite{58.7}). 
 Above conditional stability estimate is, however, not fully satisfactory. The bound for $z$ is in terms of its $L_2(G)$-norm, but boundedness of the right-hand side requires $z \in H^1(\Omega)$.
 In the derived error estimates for the numerical schemes this is reflected by the fact that these error estimates are obtained for the $L_2(G)$-norm, whereas they require approximation properties of the applied trial space in the $H^1(\Omega)$-norm (see \cite[Sect.~2.2]{35.9298}).

\subsection{Ultra-weak variational formulations} 
In the current work we will therefore consider \emph{ultra-weak} variational formulations of the UC and Cauchy problems. Considering here again the UC problem, another integration-by-parts shows that
$$
\left\{\begin{array}{r@{}c@{}ll}
-\int_\Omega u \triangle v\,dx &\,\,\,=\,\,\,& \ell(v)& (v \in H^2_0(\Omega)),\\
u|_{\omega}&\,\,\,=\,\,\,&q & \text{on }\omega.
\end{array}
\right.
$$
Redefining $A \in \cL(L_2(\Omega), H^{-2}(\Omega) \times L_2(\omega))$ by $
 Az :=\big(v \mapsto \int_\Omega -z \triangle v\,dx,z|_{\omega}\big)$, we will show that \emph{under the condition that} $\|z\|_{L_2(\Omega)} \lesssim 1$, it holds that
$$
 \|z\|_{L_2(G)} \lesssim \|A z\|_{H^{-2}(\Omega) \times L_2(\omega)}^\alpha \quad (\|A z\|_{H^{-2}(\Omega) \times L_2(\omega)} \rightarrow 0).
 $$
So we will obtain the qualitatively same conditional stability estimate, in particular with the same value of $\alpha$, but in terms of this \emph{weaker} $\|A z\|_{H^{-2}(\Omega) \times L_2(\omega)}$-norm. A similarly improved conditional stability estimate will be shown for the Cauchy problem.

\subsection{Regularized least squares approximations}
In \cite{58.7}, a regularized least squares approximation scheme was proposed for general ill-posed, but conditional stable PDE problems. For the ultra-weak variational formulation of the UC problem, a family $(X^\delta)_{\delta \in I}$ of finite dimensional subspaces of $L_2(\Omega)$, and an $\eps \geq 0$, it reads as finding
$$
u_\eps^\delta:=\argmin_{z \in X^\delta} \|A z-(\ell,q)\|_{H^{-2}(\Omega) \times L_2(\omega)}^2+\eps^2\|z\|_{L_2(\Omega)}^2.
$$
Since the $\|\bigcdot\|_{H^{-2}(\Omega)}$-norm cannot be evaluated, it is replaced by the discretized dual norm $\sup_{0 \neq v \in Y^\delta} \frac{|\bigcdot(v)|}{\|v\|_{H^2(\Omega)}}$, where $(Y^\delta)_{\delta \in I}$ is a family of finite dimensional subspaces of $H^2_0(\Omega)$ such that the pair $(X^\delta,Y^\delta)$ is uniformly stable. The latter in the sense that, for arguments $z \in X^\delta$, the discretized dual norm of $v \mapsto \int_\Omega z \triangle v\,dx$ is equivalent to its true $H^{-2}(\Omega)$-norm, uniformly in $\delta \in I$.
The general theory from \cite{58.7} shows that for a suitable choice of $\eps>0$,
\be \label{eq:40}
\|u-u_\eps^\delta\|_{L_2(G)} \lesssim \big(\|Au-(\ell,q)\|_{H^{-2}(\Omega) \times L_2(\omega)}+\min_{z \in X^\delta}\|u-z\|_{L_2(\Omega)}\big)^\alpha
\ee
for the right hand side tending to $0$. 
Compared to such an estimate for the standard variational formulation, here both the \emph{consistency error} component $-\triangle u-\ell$, and the best approximation error are measured in \emph{weaker norms}. The other consistency error component $u|_\omega-q$ is measured in $\|\bigcdot\|_{L_2(\omega)}$ for both formulations.

What remains is the \emph{construction} of uniformly stable pairs of `trial' and `test' spaces. For $(\tria^\delta)_{\delta \in I}$ being a family of conforming, uniformly shape regular triangulations of a polygon $\Omega \subset \R^2$, we show that such a stable pair is given by $X^\delta$ being the space of piecewise constants, and $Y^\delta$ being the \emph{Hsieh–Clough–Tocher} (HCT) macro finite element space, both w.r.t.~$\tria^\delta$.

\subsection{Nonconforming test spaces}
The use of $C^1$-finite element test spaces is the prize to be paid for the ultra-weak variational formulation. Since 
$C^1$-finite element are somewhat complicated to implement, we also study the use of \emph{nonconforming} test spaces.
Under a (generalized) uniform stability condition for the pair $(X^\delta,Y^\delta)$, we show a bound similar to \eqref{eq:40}, where the upper bound contains an additional `data oscillation term'.
For $X^\delta$ being the space of piecewise constants, and $Y^\delta$ the Morley finite element space, both w.r.t.~$\tria^\delta$, we show that this uniform stability condition is satisfied.
Assuming that $\ell \in L_2(\Omega)$, it is shown  that data oscillation is of higher order than the error of best approximation.
Alternatively, general $\ell \in H^{-2}(\Omega)$ can be allowed when the nonconforming test functions are smoothed by a suitable companion operator before being submitted to $\ell$. In this case the qualitatively  same error estimates are obtained as for conforming test functions, so without any oscillation term. \medskip

The results concerning the least squares approximation with conforming or nonconforming test spaces apply analogously to the ultra-weak formulation of the Cauchy problem. For this problem the ultra-weak formulation has the additional advantage that both Neumann and Dirichlet boundary conditions are natural ones, so that their imposition will be nearly effortless.

\subsection{Organization} In Sect.~\ref{sec:2}, conditional stability estimates are derived for ultra-weak variational formulations of the UC and Cauchy problems. Sect.~\ref{sec:3}--\ref{sec:6} are devoted to least squares approximations of the UC problem, where in Sect.~\ref{sec:7} the mostly minor adaptations are discussed for the application to the Cauchy problem. 

In Sect.~\ref{sec:3} the general least squares approach for conditional stable PDEs from \cite{58.7} is applied to the UC problem. The non-computable dual norm in the regularized least squares functional is replaced by a discretized dual norm, which leads to a mixed formulation. 
Under a uniform inf-sup stability condition on pairs of trial- and test-spaces, it leads to the error estimate \eqref{eq:40} where the value of $\alpha$ is shown to be optimal.
It is shown that the data approximation error and the error of best approximation are measured in weaker, and thus more favourable norms than with the standard variational formulation. 
In Sect.~\ref{sec:4}, the aforementioned uniform inf-sup stability condition is verified pairs of piecewise constants and HTC finite element spaces w.r.t.~ to common, uniformly shape regular triangulations.

In Sect.~\ref{sec:5} the analysis from \cite{58.7} is generalized to the use of nonconforming test spaces. The `medius analysis' of the effect of nonconforming spaces from \cite{77.4} for elliptic problem is generalized to mixed formulations. In Sect.~\ref{sec:6} it is shown that pairs of piecewise constants and Morley finite element spaces satisfy the necessary generalized uniform inf-sup stability condition.
This medius analysis requires the construction of a bounded mapping of the nonconforming finite element space into a conforming relative.
Taking for this latter space the HCT space, it is shown that such a `companion operator' exists with the additional property of being a right-inverse of the Fortin operator. As a consequence, when the nonconforming test functions are smoothed by this companion operator before being submitted to the forcing term, the qualitatively same error estimates are obtained as for conforming test functions.

In Sect.~\ref{sec:8} numerical results are presented, and a conclusion is formulated in Sect.~\ref{sec:9}.

%%%%%%%%%%%%%%
\section{Conditional stability} \label{sec:2}
%%%%%%%%%%%%%%
A key ingredient in deriving conditional stability results is the following ``propagation of smallness'' result for harmonic functions \cite[Thm.~5.1]{10.1}, which is a corollary of the ``three balls estimate'', e.g., see \cite[Thm.~2.1]{10.1}.

\begin{theorem}[Propagation of smallness] \label{RSprop:1}
For a bounded domain $\Omega \subset \R^n$, and subdomains $\emptyset \neq \omega \Subset G \Subset \Omega$, there exists an $\alpha \in (0,1)$ such that\footnote{\label{voetje} Obviously, for $G=\omega$ the result holds true for $\alpha=1$.} for all harmonic $w \in L_2(\Omega)$,
\be \label{RSeq:8}
\|w\|_{L_2(G)} \lesssim \|w\|_{L_2(\omega)}^\alpha \|w\|_{L_2(\Omega)}^{1-\alpha}.
\ee
\end{theorem}

\subsection{Conditional stability for the UC problem}
Recall that we consider the \emph{ultra-weak} variational formulation of the UC problem of finding $u \in L_2(\Omega)$ such that
\be \label{eq:42}
\left\{\begin{array}{r@{}c@{}ll}
-\int_\Omega u \triangle v\,dx &\,\,\,=\,\,\,& \ell(v)& (v \in H^2_0(\Omega)),\\
u|_{\omega}&\,\,\,=\,\,\,&q & \text{on }\omega.
\end{array}
\right.
\ee

\begin{theorem} [Conditional stability UC ultra-weak] \label{RSprop:2}
With $A \in \cL\big(L_2(\Omega),H^{-2}(\Omega) \times L_2(\omega)\big)$ defined by
$$
A :=(B,C),\quad (Bz)(v):= -\int_\Omega  z \triangle v\,dx \quad(v \in H^2_0(\Omega)),\quad Cz:=z|_\omega, 
$$
for bounded domains $\emptyset \neq \omega \subseteq G \Subset \Omega$, and $\alpha \in (0,1]$ from Theorem~\ref{RSprop:1}, it holds that
$$
\|z\|_{L_2(G)} \!\lesssim \! \|Az\|_{H^{-2}(\Omega) \times L_2(\omega)}^\alpha \!\big(\|z\|_{L_2(\Omega)}\!+ \!\|Az\|_{H^{-2}(\Omega) \times L_2(\omega)}\big)^{1-\alpha} \quad (z \in L_2(\Omega)).
$$
\end{theorem}

\begin{proof} It holds that
\be \label{RSeq:2}
\|\bigcdot\|_{H^2(\Omega)} \eqsim \|\triangle \bigcdot\|_{L_2(\Omega)} \,\,\,\text{ on } H^2_0(\Omega).
\ee
% In 2D geldt $|\bigcdot|_{H^2(\Omega)} = \|\triangle \bigcdot\|_{L_2(\Omega)}$, maar dat gaat niet door in n>2.
Indeed, consider a ball $Z \supset \Omega$. From the $H^2$-regularity of the Poisson problem  on $Z$ with homogeneous Dirichlet boundary conditions, for $\varphi \in \cD(\Omega)$ we have
$$
\|\varphi\|_{H^2(\Omega)} = \|\varphi\|_{H^2(Z)} \lesssim \|\triangle \varphi\|_{L_2(Z)}=\|\triangle \varphi\|_{L_2(\Omega)}.
$$

Let $\cH(\Omega):=\ran \triangle|_{H_0^2(\Omega)} \subset L_2(\Omega)$. As a consequence of \eqref{RSeq:2}, we have that $\cH(\Omega)$ is closed, and therefore is a Hilbert space.\footnote{From $H^2_0(\Omega)$ being a proper subspace of $H^1_{\triangle,0}(\Omega):=\{z \in H^1_0(\Omega)\colon \triangle z \in L_2(\Omega)\}$ equipped with graph norm, and $\triangle \in \cL(H^1_{\triangle,0}(\Omega),L_2(\Omega))$ being boundedly invertible, we infer that $\cH(\Omega)$ is a proper subspace of $L_2(\Omega)$.}
The bilinear form $(z,v) \mapsto -\int_\Omega z \triangle v\,dx$ on $\cH(\Omega) \times H^2_0(\Omega)$ is bounded,
\begin{alignat*}{2}
\sup_{0 \neq v \in H^2_0(\Omega)}\frac{\int_\Omega z \triangle v\,dx}{\|v\|_{H^2(\Omega)}} &\eqsim
\sup_{0 \neq v \in H^2_0(\Omega)}\frac{\int_\Omega z \triangle v\,dx}{\|\triangle v\|_{L_2(\Omega)}} =\|z\|_{L_2(\Omega)} &&\quad(z \in \cH(\Omega)),\intertext{and}
\sup_{0 \neq z \in \cH(\Omega)} \frac{\int_\Omega z \triangle v\,dx}{\|z\|_{L_2(\Omega)}} & = \|\triangle v\|_{L_2(\Omega)} \eqsim \|v\|_{H^2(\Omega)} &&\quad (v \in H^2_0(\Omega)).
\end{alignat*}
We conclude that for $z \in L_2(\Omega)$, there exists a unique $\widehat{z} \in \cH(\Omega)$ with
$$
-\int_\Omega \widehat{z}  \triangle v\,dx=-\int_\Omega z  \triangle v\,dx\quad (v \in H^2_0(\Omega)),
$$
and $\|\widehat{z}\|_{L_2(\Omega)} \lesssim \|Bz\|_{H^{-2}(\Omega)}$.

We set $\widetilde{z}:=z-\widehat{z}$. Then $\int_\Omega \widetilde{z} \triangle \varphi\,dx=0$ ($\varphi \in \cD(\Omega)$), and according to Weyl's lemma (\cite{311.7}) this means that $\widetilde{z}$ is harmonic.

As before, repeated applications of the triangle-inequality, and an application of Theorem~\ref{RSprop:1} to $\widetilde{z}$ show that
\begin{align*}
& \|z\|_{L_2(G)}  \leq \|\widehat{z}\|_{L_2(G)}+\|\widetilde{z}\|_{L_2(G)} \lesssim \|Bz\|_{H^{-2}(\Omega)}+\|\widetilde{z}\|_{L_2(\omega)}^\alpha  \|\widetilde{z}\|_{L_2(\Omega)}^{1-\alpha}\\
&  \lesssim \|Bz\|_{H^{-2}(\Omega)}+(\|z\|_{L_2(\omega)}+\|\widehat{z}\|_{L_2(\omega)})^\alpha (\|z\|_{L_2(\Omega)}+\|\widehat{z}\|_{L_2(\Omega)})^{1-\alpha}\\
&  \lesssim (\|C z\|_{L_2(\omega)}+\|Bz\|_{H^{-2}(\Omega)})^\alpha (\|z\|_{L_2(\Omega)}+\|Bz\|_{H^{-2}(\Omega)}+\|Cz\|_{L_2(\omega)})^{1-\alpha},
\end{align*}
which is the estimate to be proven.
\end{proof}

\begin{remark} \label{rem:1} The operator $A \in \cL\big(L_2(\Omega),H^{-2}(\Omega) \times L_2(\omega)\big)$ is injective. Indeed, above proof shows that $Bz=0$ implies that $z$ is harmonic. Using $z|_{\omega}=0$, Theorem~\ref{RSprop:1} now shows that $z$ vanishes on any $G  \Subset\Omega$, so that $z=0$.
\end{remark}

\subsection{Conditional stability for the Cauchy problem}
By multiplying \eqref{eq:41} with smooth $v$ with $\supp v \subset \Omega \cup \Sigma$, and by integrating-by-parts twice whilst substituting the Neumann and Dirichlet boundary conditions on $\Sigma$, we obtain the problem to find $u \in L_2(\Omega)$ such that
$$
(Bu)(v):=-\int_\Omega u \triangle v\,dx=f(v):=\ell(v)+\int_{\Sigma} \psi v-g \partial_n v\,ds \quad (v \in H^2_{0,\Sigma^c}(\Omega)),
$$
where $\Sigma^c:=\partial\Omega \setminus \overline{\Sigma}$, and $H^2_{0,\Sigma^c}(\Omega):=\{v \in C^\infty(\Omega) \cap H^2(\Omega)\colon \supp v \cap \Sigma^c=\emptyset\}$.

\begin{theorem}[Conditional stability Cauchy ultra-weak] \label{thm:condstabCauchy} Let $\Omega \subset \R^n$ be a bounded domain, $\Sigma \neq \emptyset$ an open, connected Lipschitz subset of $\partial\Omega$, and $G \subset \Omega$ a domain that has a positive distance to $\Sigma^c$.
Then there exists an $\alpha \in (0,1)$ such that
$$
\|z\|_{L_2(G)} \lesssim \|B z\|^\alpha_{(H^2_{0,\Sigma^c}(\Omega))'}(\|z\|_{L_2(\Omega)}+ \|B z\|_{(H^2_{0,\Sigma^c}(\Omega))'})^{1-\alpha}\quad(z \in L_2(\Omega)).
$$
\end{theorem}

\begin{proof} When $\overline{G} \cap \Sigma=\emptyset$, enlarge $G$ to a domain $\widehat{G}$ with $\overline{\widehat{G}} \cap \Sigma \neq \emptyset$, where, just as $G$, $\widehat{G}$ has a positive distance to $\Sigma^c$.

Let $\widetilde{\Sigma} \neq \emptyset $ be a relatively open and connected subset of $\Sigma$ with $\Sigma \cap \widehat{G} \subset \widetilde{\Sigma} \Subset \Sigma$.
With $\widetilde{\Sigma}(\eps):=\{x \in \R^n\colon \dist(x,\widetilde{\Sigma})<\eps\}$, let $r>0$ be such that
$\widetilde{\Omega}:=\Omega \cup \widetilde{\Sigma}(r)$ has a positive distance to $\Sigma^c$.
Set $\widetilde{G}:=\widehat{G} \cup \widetilde{\Sigma}(r/2)$. Then $\widetilde{G}$ is connected, and from $\widetilde{G}\setminus \widetilde{\Sigma}(r/2) \Subset \Omega$ and $\widetilde{\Sigma}(r/2) \Subset \widetilde{\Sigma}(r)$ it follows that $\widetilde{G} \Subset \widetilde{\Omega}$.
Finally, let $\widetilde{\omega}$ be an open and connected set with $\widetilde{\omega} \Subset \widetilde{G} \setminus \overline{\Omega}$, see Figure~\ref{fig:geo}.
\begin{figure}[h]
  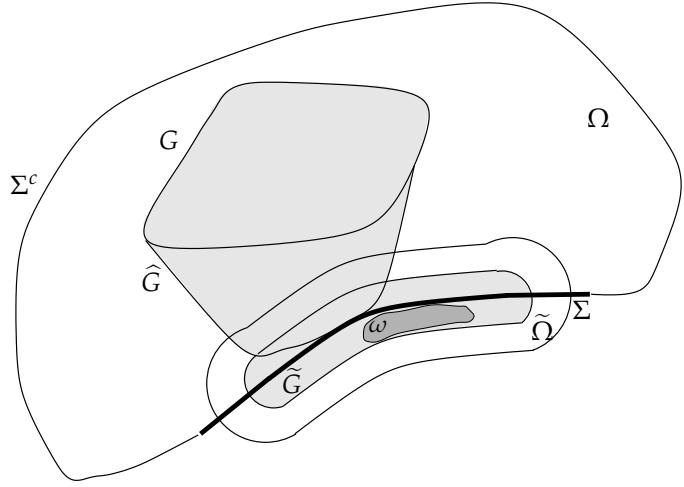
    \caption{Illustration with proof of Theorem~\ref{thm:condstabCauchy}}
  \label{fig:geo}
\end{figure}

For $z \in L_2(\Omega)$, let $\widetilde{z} \in L_2(\widetilde{\Omega})$ be the zero extension of $z$ to $\widetilde{\Omega}$.
Setting $\widetilde{B} \in \cL(L_2(\widetilde{\Omega}),H^{-2}(\widetilde{\Omega}))$ by $(\widetilde{B} w)(v):=-\int_{\widetilde{\Omega}} w \triangle v\,dx$, using that $\widetilde{z}|_{\widetilde{\omega}}=0$ an application of Theorem~\ref{RSprop:2} with $(\omega,G,\Omega,B)$ reading as $(\widetilde{\omega},\widetilde{G},\widetilde{\Omega},\widetilde{B})$ shows that for some constant $\alpha \in (0,1)$,
$$
%\|z\|_{L_2(\widetilde{G})} \leq 
 \|\widetilde{z}\|_{L_2(\widetilde{G})} \lesssim \|\widetilde{B} \widetilde{z}\|_{H^{-2}(\widetilde{\Omega})}^\alpha \big(\|\widetilde{z}\|_{L_2(\widetilde{\Omega})}+\|\widetilde{B} \widetilde{z}\|_{H^{-2}(\widetilde{\Omega})}\big)^{1-\alpha}.
$$
The proof is completed by $\|z\|_{L_2(G)} \leq \|\widetilde{z}\|_{L_2(\widetilde{G})}$, $\|\widetilde{z}\|_{L_2(\widetilde{\Omega})}=\|z\|_{L_2(\Omega)}$, and
$$
\sup_{0 \neq \widetilde{v} \in H_0^2(\widetilde{\Omega})}\!\! \frac{\int_{\widetilde{\Omega}} \widetilde{z} \triangle \widetilde{v} \, dx}{\|\widetilde{v}\|_{H^2(\widetilde{\Omega})}} \leq \hspace*{-.5em}
\sup_{ \{\widetilde{v} \in H_0^2(\widetilde{\Omega})\colon \widetilde{v}|_\Omega \neq 0\}}\!\!  \frac{\int_{\Omega} z \triangle \widetilde{v}|_{\Omega} \, dx}{\|\widetilde{v}|_\Omega\|_{H^2(\Omega)}} \leq
\hspace*{-.5em} \sup_{0 \neq v \in H_{0,\Sigma^c}^2(\Omega)} \!\!  \frac{\int_{\Omega} z \triangle v \, dx}{\|v\|_{H^2(\Omega)}},
$$
by $\Sigma^c \subset \partial\widetilde{\Omega}$.
\end{proof}

\begin{remark} \label{rem:2} The operator $B \in \cL\big(L_2(\Omega),(H^2_{0,\Sigma^c}(\Omega))'\big)$ is injective. Indeed, above proof shows that $Bz=0$ implies that $\widetilde{B} \widetilde{z}=0$ and $\widetilde{z}|_{\widetilde{\omega}}=0$. The arguments from Remark~\ref{rem:1} now show that $\widetilde{z}=0$ and thus $z=0$.
\end{remark}

\newpage

\section{Least squares approximations for the UC problem}  \label{sec:3}
We consider the UC problem \eqref{eq:42}, and postpone the discussion of the (small) adaptations needed for the Cauchy problem to Sect.~\ref{sec:7}. Our goal is to (approximately) recover $u \in L_2(\Omega)$ from the equation $Au=(Bu,Cu)=(\ell,q)$, with $A$ defined in Theorem~\ref{RSprop:2}. For inexact data $(\ell,q)$, generally this equation has no solution which is why we search a least squares approximation. Moreover, following the general approach from \cite{58.7} for conditionally stable PDEs, we consider a regularized least squares functional. The regularization is designed to ensure control of the $L_2(\Omega)$-norm of the obtained approximation by an absolute multiple of $\|u\|_{L_2(\Omega)}$, so that we can successfully apply the conditional stability estimate to its error.

Specifically, let $(X^\delta)_{\delta \in I}$ be some (infinite) family of finite dimensional subspaces of $L_2(\Omega)$.
For $\delta \in I$ and a suitable $0 \leq \eps \lesssim 1$, we approximate $u \in L_2(\Omega)$ by the minimizer over $z \in X^\delta$ of the \emph{regularized least squares functional}
\be \label{eq:3}
z \mapsto \|A z-(\ell,q)\|_{H^{-2}(\Omega) \times L_2(\omega)}^2+\eps^2\|z\|_{L_2(\Omega)}^2.
\ee

To deal with the non-computable norm on $H^{-2}(\Omega)$, let $(Y^\delta)_{\delta \in I}$ be a family of finite dimensional (or closed) subspaces of $H^2_0(\Omega)$
for which 
\be \label{eq:infsup}
\varrho:=\inf_{\delta \in I} \inf_{\{z \in X^\delta\colon Bz\neq 0\}} \sup_{0 \neq v \in Y^\delta}\frac{|(Bz)(v)|}{\|B z\|_{H^{-2}(\Omega)}\|\triangle v\|_{L_2(\Omega)}}>0.
\ee
Notice that here we have used \eqref{RSeq:2}, i.e., $\|\bigcdot\|_{H^2(\Omega)} \eqsim \|\triangle \bigcdot\|_{L_2(\Omega)}$ on $H^2_0(\Omega)$.
We will refer to $X^\delta$ and $Y^\delta$ as \emph{trial} and \emph{test spaces}.
By replacing \mbox{$\|\bigcdot\|_{H^{-2}(\Omega)}$} by the discretized dual norm $\sup_{0 \neq v \in Y^\delta} \frac{|\bigcdot\,(v)|}{\|\triangle v\|_{L_2(\Omega)}}$ in \eqref{eq:3}, an approximation to $u$ is obtained that is qualitatively equally good. Specifically, we have the following result.

\begin{theorem}[{\cite[Thm.~4.1]{58.7}}] \label{thm:ls}
Assume \eqref{eq:infsup}, and let $(\ell,q) \in H^{-2}(\Omega) \times L_2(\omega)$. Then\footnote{Because $A$ is injective, the proof given in \cite{58.7} for $\eps>0$ extends to $\eps \geq 0$. Cf.~also the forthcoming generalization Theorem~\ref{thm:lsnonconf}.}
\be \label{eq:4}
u_\eps^\delta:=\argmin_{z \in X^\delta} \big\{\sup_{0 \neq v \in Y^\delta}\frac{|(B z-\ell)(v)|^2}{\|\triangle v\|_{L_2(\Omega)}^2}+\|Cz-q\|_{L_2(\omega)}^2+\eps^2\|z\|_{L_2(\Omega)}^2 \big\}
\ee
satisfies, for $u \in L_2(\Omega)$,
\be \label{eq:14}
\|A(u-u_\eps^\delta)\|_{H^{-2}(\Omega)\times L_2(\omega)}+\eps\|u-u_\eps^\delta\|_{L_2(\Omega)} \lesssim \cE_{\rm data}+\cE_{\rm appr}(\delta) +\eps\|u\|_{L_2(\Omega)},
\ee
where
$$
\cE_{\rm data}:=\|Au-(\ell,q)\|_{H^{-2}(\Omega) \times L_2(\omega)},\quad \cE_{\rm appr}(\delta):=\min_{z \in X^\delta} \|u-z\|_{L_2(\Omega)}.
$$
\end{theorem}

\begin{corollary} \label{RScorol:1} In the setting of  Theorem~\ref{thm:ls},  let $\eps>0$ be such that
\be \label{RSeq:eps}
\eps \|u\|_{L_2(\Omega)} \eqsim \cE_{\rm data}+\cE_{\rm appr}(\delta).
\ee
Then with $\alpha \in (0,1]$ from Theorem~\ref{RSprop:1}, it holds that
\be \label{RSeq:error-estimate}
\|u-u_\eps^\delta\|_{L_2(G)}  \lesssim 
\big( \cE_{\rm data}+\cE_{\rm appr}(\delta)\big)^\alpha \max\big(\cE_{\rm data}+\cE_{\rm appr}(\delta),\|u\|_{L_2(\Omega)}\big)^{1-\alpha}.
\ee
\end{corollary}

\begin{proof} Substituting \eqref{RSeq:eps} in \eqref{eq:14} shows that
$$
\left\{
\begin{array}{rcl}
\|u-u_\eps^\delta\|_{L_2(\Omega)} & \lesssim&  \|u\|_{L_2(\Omega)},\\
\|A(u-u_\eps^\delta)\|_{H^{-2}(\Omega)\times L_2(\omega)} & \lesssim& \cE_{\rm data}+\cE_{\rm appr}(\delta),
\end{array}
\right.
$$
so that an application of the conditional stability result from Theorem~\ref{RSprop:2}  gives the result.
\end{proof}

\begin{remark}[Optimality of $\alpha$ in \eqref{RSeq:error-estimate}] \label{RSrem:opt}
For $(X^\delta)_{\delta \in I}$ with $\overline{\cup_{\delta \in I} X^\delta}=L_2(\Omega)$, an estimate as \eqref{RSeq:error-estimate} 
with an exponent\footnote{In two or three dimensions, the value of $\widehat{\alpha}$ is known for $\omega$, $G$, and $\Omega$ being balls centered around some common point, see \cite[Thm.~1]{35.9298}.} $\widetilde{\alpha}>\widehat{\alpha}:=\sup\{\alpha\colon \eqref{RSeq:8} \text{ is valid for } \alpha\}$ cannot be valid for \emph{any} scheme $(\ell,q)\mapsto (u^\delta)_{\delta \in I} \subset \prod_{\delta \in I}X^\delta$ that maps zero data $\ell=0=q$ onto $u^\delta \equiv 0$. Indeed, suppose such a scheme does exist, and consider $\ell=0=q$ and a harmonic $u \in L_2(\Omega)$.
% Een harmonische functie op $\Omega$ is niet noodzakelijk in $L_2(\Omega)$. Immers neem een harmonische functie niet in $H^2(\Omega)$ en bekijk z'n tweede orde partiele afgeleiden, welke ook harmonisch zijn.
Then
$$
\|u\|_{L_2(G)}  \lesssim \inf_{\delta \in I} (\|u\|_{L_2(\omega)} + \min_{z \in X^\delta} \|u-z\|_{L_2(\Omega)})^{\tilde{\alpha}}
\|u\|_{L_2(\Omega)}^{1-\tilde{\alpha}}=\|u\|_{L_2(\omega)}^{\tilde{\alpha}}
\|u\|_{L_2(\Omega)}^{1-\tilde{\alpha}},
$$
in contradiction to the definition of $\widehat{\alpha}$.
\end{remark}

The approximation $u_\eps^\delta$ defined in \eqref{eq:4} is \emph{computed} as the first component of the pair $(u_\eps^\delta,v_\eps^\delta) \in X^\delta \times Y^\delta$ that solves the \emph{saddle point problem}
\be \label{eq:saddle1}
\left\{
\begin{array}{@{}c@{}c@{}c@{}cr}
\langle \triangle v_\eps^\delta, \triangle \tilde{v}\rangle_{L_2(\Omega)} \hspace*{-2em}& +  (B u_\eps^\delta)(\tilde{v}) & =\,\, & \ell(\tilde{v}) & (\tilde{v}\in Y^\delta),\\
(B \tilde{z})(v_\eps^\delta) & -\langle C u_\eps^\delta,C \tilde{z}\rangle_{L_2(\omega)}-\eps^2\langle u_\eps^\delta,\tilde{z}\rangle_{L_2(\Omega)} & =\,\, &- \langle q,C \tilde{z} \rangle_{L_2(\omega)} & (\tilde{z} \in X^\delta),
\end{array}
\right.
\ee
(see e.g.~\cite{58.7}).
\medskip

Finally, the uniform inf-sup or `LBB'-stability condition \eqref{eq:infsup} is \emph{equivalent} to existence of uniformly bounded `Fortin' operators. Precisely (e.g.~\cite[Prop.~5.1]{249.992}), assuming $\ran B|_{X^\delta} \neq \{0\}$ and $Y^\delta \neq \{0\}$, let
\be \label{fortin}
Q^\delta \in \cL(H^2_0(\Omega),Y^\delta) \text{ with } (B X^\delta)\big(\ran (\identity - Q^\delta)\big)=0.
\ee
Then $\varrho^\delta := \inf_{\{z \in X^\delta\colon B z \neq 0\}} \sup_{0 \neq v \in Y^\delta} \frac{|(B z)(v)|}{\|B z\|_{H^{-2}(\Omega)} \|\triangle v\|_{L_2(\Omega)}} \geq \|Q^\delta\|_{\cL(H_0^2(\Omega),H_0^2(\Omega))}^{-1}$.

Conversely, if $\varrho^\delta>0$, then there exists a $Q^\delta$ as in \eqref{fortin}, being even a projector onto $Y^\delta$, with
$\|Q^\delta\|_{\cL(H_0^2(\Omega),H_0^2(\Omega))} = 1/\varrho^\delta$.

\begin{remark}[Discussion] To compare with results from the literature, let us choose the regularization parameter $\eps>0$ such that $\eps\|u\|_{L_2(\Omega)} \eqsim \cE_{\rm appr}(\delta)$, which corresponds to the choice made in Corollary~\ref{RScorol:1} as long as $\cE_{\rm appr}(\delta) \gtrsim \cE_{\rm data}$. With this choice of $\eps$, \eqref{eq:14} shows 
$$
\left\{
\begin{array}{rcl}
\|u-u_\eps^\delta\|_{L_2(\Omega)} & \lesssim& \big(1+\frac{ \cE_{\rm data}}{\cE_{\rm appr}(\delta)}\big) \|u\|_{L_2(\Omega)},\\
\|A(u-u_\eps^\delta)\|_{H^{-2}(\Omega)\times L_2(\omega)} & \lesssim&\cE_{\rm data}+\cE_{\rm appr}(\delta).
\end{array}
\right.
$$
Assuming that $\cE_{\rm data}+\cE_{\rm appr}(\delta) \lesssim \|u\|_{L_2(\Omega)}$, substituting these bounds in the conditional stability result from Theorem~\ref{RSprop:2} gives
\begin{align} \nonumber
\|u-u_\eps^\delta\|_{L_2(G)} &\lesssim \big(\cE_{\rm data}+\cE_{\rm appr}(\delta)\big)^\alpha \Big(\big(1+\frac{ \cE_{\rm data}}{\cE_{\rm appr}(\delta)}\big) \|u\|_{L_2(\Omega)}+\cE_{\rm data}+\cE_{\rm appr}(\delta)\Big)^{1-\alpha}\\ \nonumber
&\eqsim \big(\cE_{\rm data}+\cE_{\rm appr}(\delta)\big)^\alpha
\big(1+\frac{ \cE_{\rm data}}{\cE_{\rm appr}(\delta)}\big)^{1-\alpha} \|u\|_{L_2(\Omega)}^{1-\alpha}\\ \label{RSeq:9}
&\eqsim \Big(\cE_{\rm appr}(\delta)^\alpha +\frac{\cE_{\rm data}}{\cE_{\rm appr}(\delta)^{1-\alpha}}\Big)\|u\|_{L_2(\Omega)}^{1-\alpha}.
\end{align}
This upper  bound is proportional to $\cE_{\rm appr}(\delta)^\alpha \|u\|_{L_2(\Omega)}^{1-\alpha}$ as long as $\cE_{\rm appr}(\delta) \gtrsim \cE_{\rm data}$, and thus 
is proportional to $(\cE_{\rm data})^\alpha \|u\|^{1-\alpha}$ when $\cE_{\rm appr}(\delta) \eqsim \cE_{\rm data}$, whereas it increases unboundedly for
 $\cE_{\rm appr}(\delta) \downarrow 0$ when $\cE_{\rm appr}(\delta) \lesssim \cE_{\rm data}$.
 \footnote{Although it results in a (qualitatively) equal minimal upper bound for $\|u-u_\eps^\delta\|_{L_2(G)}$, our choice in Corollary~\ref{RScorol:1} of $\eps \|u\|_{L_2(\Omega)} \eqsim \cE_{\rm appr}(\delta)+\cE_{\rm data}$ prevents this latter increase.}

Still for the purpose of comparison, let us consider $X^\delta$ to be a finite element space of degree $k$ w.r.t.~a \emph{quasi-uniform} partition of $\Omega$ with mesh-size $h_\delta$ into uniformly shape regular $n$-simplices. 

For the primal-dual weakly consistent regularized finite element method studied in \cite{35.9298}, where the UC problem is considered with forcing term $\ell=0$, it was shown that
\be \label{RSeq:10}
\|u-u^\delta\|_{L_2(G)} \lesssim h_\delta^{\alpha k}\|u\|_{H^{k+1}(\Omega)}+h_\delta^{(\alpha-1)k} \|u|_\omega-q\|_{L_2(\omega)},
\ee
and so
$$
\|u-u^\delta\|_{L_2(G)} \lesssim \|u|_\omega-q\|_{L_2(\omega)}^\alpha \|u\|_{H^{k+1}(\Omega)}^{1-\alpha}
$$
for the minimum
 $$
 h_\delta
 \eqsim \Big(\frac{\|u|_\omega-q\|_{L_2(\omega)}}{\|u\|_{H^{k+1}(\Omega)}}\Big)^{\frac{1}{k}}.
 $$
 It was shown that the value of $\alpha$ in the exponents in \eqref{RSeq:10} is optimal, but below we see that the value of $k$ in these exponents can be improved.

Indeed, using that $\cE_{\rm appr}(\delta) \lesssim h_\delta^{k+1}|u|_{H^{k+1}(\Omega)}$ when $u \in H^{k+1}(\Omega)$, and by substituting this upper bound into \eqref{RSeq:9} shows
$$
\|u-u_\eps^\delta\|_{L_2(G)} \lesssim 
\Big(\big(h_\delta^{k+1}|u|_{H^{k+1}(\Omega)}\big)^\alpha +\frac{\cE_{\rm data}}{\big(h_\delta^{k+1}|u|_{H^{k+1}(\Omega)}\big)^{1-\alpha}}\Big)\|u\|_{L_2(\Omega)}^{1-\alpha}.
$$
Solving $h_\delta^{k+1}|u|_{H^{k+1}(\Omega)} \eqsim \cE_{\rm data}$ gives
\be \label{RSeq:12}
\|u-u_\eps^\delta\|_{L_2(G)} \lesssim (\cE_{\rm data})^\alpha \|u\|_{L_2(\Omega)}^{1-\alpha}
\ee
for
$$
h_\delta \eqsim \big(\frac{\cE_{\rm data}}{|u|_{H^{k+1}(\Omega)}}\big)^{\frac{1}{k+1}}.
$$

So our least squares method gives a qualitatively similar bound, but, thanks to the use of an ultra-weak formulation, for a larger mesh-size, and so at lower computational cost (cf.~\cite[Rem.~3]{35.9298}). Alternatively, by taking finite elements of one lower degree $k-1$, we find \eqref{RSeq:12} for 
$h_\delta \eqsim \big(\frac{\cE_{\rm data}}{|u|_{H^{k}(\Omega)}}\big)^{\frac{1}{k}}$, so with a mesh-size as in \cite{35.9298}, but under a relaxed regularity condition.

Finally, notice that our least squares method is not restricted to quasi-uniform meshes, meaning that with appropriately locally refined meshes
regularity conditions can be relaxed even further. On the other hand, our least squares method requires the construction of a family of pairs of trial and test spaces that satisfies the uniform inf-sup condition \eqref{eq:infsup}.
\end{remark}

What remains to show is that our least squares method is practical by constructing test spaces, with dimensions proportional to those of the trial spaces, that give uniform inf-sup stability. This will be done in the following section for piecewise constant and continuous piecewise linear finite element trial spaces. 

\section{Examples of trial and (conforming) test spaces} \label{sec:4}
Let $\tria^\delta$ be a conforming (uniformly) shape regular triangulation of a \emph{polygon} $\Omega \subset \R^2$.
With $\cN^\delta$ ($\cN_\circ^\delta$) and $\cE^\delta$ ($\cE_\circ^\delta$) we denote the sets of its (internal) vertices and (internal) edges. For $K \in \tria^\delta$ and $e \in \cE^\delta$, we set $h_K:=\diam(K)$ and $h_e:=\diam(e)$.

First, let $X^\delta$ be the space of \emph{piecewise constants} w.r.t.~$\tria^\delta$.
For $z \in X^\delta$ and $v \in H^2_0(\Omega)$, integration-by-parts shows that
\be \label{eq:5}
(Bz)(v)=-\sum_{K \in \tria^\delta}\int_{K} z \triangle v \,dx=-\sum_{e \in \cE_\circ^\delta} \llbracket z \rrbracket_e \int_e \partial_{n_e} v\,ds.
\ee

For $Y^\delta$ we take the \emph{Hsieh–Clough–Tocher} (HCT) macro-element finite element space
w.r.t.~$\tria^\delta$, i.e.,
\be\label{eq:HCT}
Y^\delta=\big\{v \in H^2_0(\Omega)\colon v|_{K} \in \{z \in C^1(K)\colon z|_{K_i} \in P_3(K_i),\,1 \leq i \leq 3\} (K \in \tria^\delta)\big\},
\ee
where $\{K_i\colon 1 \leq i \leq 3\}$ is the subdivision of $K$ obtained by connecting the centroid of $K$ with the midpoints of its edges, see Figure~\ref{fig:1}.
\begin{figure}[h]
\includegraphics[width=4cm]{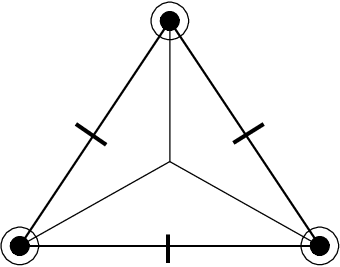}
\caption{\label{fig:1} HTC macro-element and its 12 DoF.}
\end{figure}
The vanishing trace and vanishing normal trace conditions on $\partial\Omega$ are enforced by setting the DoFs associated to points on $\partial\Omega$ to zero.

It is known how to construct a projector $P^\delta\colon H^2_0(\Omega) \rightarrow Y^\delta$ for which, for $K \in \tria^\delta$,
\be \label{eq:6}
\|(\identity -P^\delta)v\|_{H^k(K)} \lesssim h_K^{\ell-k} |v|_{H^\ell(\cup\omega(K))} \,\,\,(0 \leq k \leq 2 \leq \ell\leq 4),\hspace*{-0.5em}
\ee
where $\omega(K):=\{K' \in \tria^\delta\colon K \cap K' \neq \emptyset\}$ (e.g.~see the adaptation of the Scott-Zhang projector for Lagrange elements to Hermite elements in \cite{75.51}).

To construct from $P^\delta$ a valid Fortin operator, in view of \eqref{eq:5} we have to adapt it such that it preserves averages of normal derivatives over $e \in \cE_\circ^\delta$.
Let $b_e$ denote the nodal basis function in $Y^\delta$ that corresponds to the DoF associated to the midpoint $m_e$ of $e$.
We have $\partial_{n_e} b_e \in P_2(e)$, $\partial_{n_e} b_e(m_e)=1$, and $\partial_{n_e} b_e$ vanishes at the endpoints of $e$. So
$$
\int_e\partial_{n_e} b_e\,ds = \tfrac23 h_e. 
$$
Furthermore, both $b_e$ and its normal derivative vanish at all edges of $\tria^\delta$ other than $e$, and
\be  \label{eq:8}
|b_e|_{H^k(\Omega)} \eqsim h_e^{2-k} \quad (0 \leq k \leq 2).
\ee
Setting
$$
Q^\delta v:=P^\delta v+\sum_{e \in \cE_\circ^\delta} \frac{\int_e \partial_{n_e}(\identity-P^\delta)v\,ds}{\int_e \partial_{n_e} b_e\,ds} b_e,
$$
we have
\be \label{eq:9}
\int_e \partial_{n_e} Q^\delta v\,ds =\int_e \partial_{n_e} v\,ds \quad (e \in \cE_\circ^\delta).
\ee
From
\be \nonumber
\begin{split}
|\int_e \partial_{n_e}(\identity-P^\delta)v\,ds| & \lesssim h_e^{\frac12} \|\partial_{n_e}(\identity-P^\delta)v\|_{L_2(e)} \\
& \lesssim
\|(\identity-P^\delta)v\|_{H^1(K)}+h_e|(\identity-P^\delta)v|_{H^2(K)},
\end{split}
\ee
by a trace inequality, where $K \in \tria^\delta$ contains the edge $e$, in combination with \eqref{eq:6}--\eqref{eq:8}, we conclude that
$$
\|(\identity -Q^\delta)v\|_{H^k(K)} \lesssim h_K^{\ell-k} |v|_{H^\ell(\cup\omega(K))} \quad(0 \leq k \leq 2 \leq \ell\leq 4),
$$
so that in particular $\|Q^\delta\|_{\cL(H_0^2(\Omega),H_0^2(\Omega))} \lesssim1$. Together with \eqref{eq:9} and \eqref{eq:5} it shows that $Q^\delta$ is valid Fortin operator, so that the family of pairs $(X^\delta,Y^\delta)_{\delta \in I}$ satisfies the uniform inf-sup condition \eqref{eq:infsup}, and Theorem~\ref{thm:ls} and Corollary~\ref{RScorol:1} are applicable.
\medskip

As a second option, let $X^\delta$ be the space of \emph{continuous piecewise linears} w.r.t.~$\tria^\delta$.
For $z \in X^\delta$ and $v \in H^2_0(\Omega)$, integration-by-parts shows that
$$
(Bz)(v)=\sum_{e \in \cE_\circ^\delta} \llbracket \partial_{n_e} z \rrbracket_e \int_e v\,ds.
$$
So a valid Fortin operator should preserve averages over $e \in \cE_\circ^\delta$.

For each $e \in \cE_\circ^\delta$, $e=K_1 \cap K_2$ with $K_1,K_2 \in \tria^\delta$, let $b_e \in H^2_0(K_1 \cup K_2)$ with $\int_e b_e \,ds \gtrsim h_e$, $|b_e|_{H^k(\Omega)} \lesssim h_e^{1-k}$ ($0 \leq k \leq 2$). Such a $b_e$ can e.g.~be found in the HCT space w.r.t.~a refined conforming uniformly shape regular triangulation in which each edge in $\tria^\delta$ has been cut. Then
$$
\tilde{Q}^\delta v:=P^\delta v+\sum_{e \in \cE_\circ^\delta} \frac{\int_e (\identity-P^\delta)v\,ds}{\int_e b_e\,ds} b_e,
$$
satisfies $\int_e \tilde{Q}^\delta v\,ds =\int_e  v\,ds$ ($e \in \cE_\circ^\delta$), and similar arguments as applied above show
$$
\|(\identity -\tilde{Q}^\delta)v\|_{H^k(K)} \lesssim h_K^{\ell-k} |v|_{H^\ell(\cup\omega(K))} \quad(0 \leq k \leq 2 \leq \ell\leq 4),
$$
We conclude that $(X^\delta,Y^\delta+\Span\{b_e\colon e \in \cE_\circ^\delta\})_{\delta \in I}$ satisfies the uniform inf-sup condition \eqref{eq:infsup}, so that Theorem~\ref{thm:ls} and Corollary~\ref{RScorol:1} are applicable.
The number of local DoF of this extended test finite element space is $15$.
\medskip

Following above lines, also for higher order finite element trial spaces appropriate corresponding test finite element spaces in $H^2_0(\Omega)$ can be found. It is however fair to say that $C^1$-finite element spaces are not easy to implement. Therefore in the following we investigate whether nonconforming spaces can be applied instead.

\section{Nonconforming test spaces} \label{sec:5}
As before, let $(X^\delta)_{\delta \in I}$ be some (infinite) family of finite dimensional subspaces of $L_2(\Omega)$.
Let $(Y^\delta)_{\delta \in I}$ be a family of (finite dimensional) Hilbert spaces, which is \emph{not} necessarily included in $H^2_0(\Omega)$, with norms $\|\bigcdot\|_{Y^\delta}$. 
Let $E^\delta \in \cL(Y^\delta,H^2_0(\Omega))$, in the literature sometimes called a \emph{companion operator} or \emph{smoother}, with
\be \label{eq:12}
\sup_{\delta \in I} \|E^\delta\|_{\cL(Y^\delta,H^2_0(\Omega))}<\infty,
\ee
and let $B^\delta \in \cL\big(L_2(\Omega),(H^2_0(\Omega) + Y^\delta)'\big)$ be such that
\be \label{eq:15}
(B^\delta\bigcdot)(v)=(B \bigcdot)(v) \quad(v \in H^2_0(\Omega)).
\ee
%and
%\be \label{eq:11}
%\sup_{\delta \in I} \|B^\delta\|_{\cL\big(L_2(\Omega),(Y^\delta)'\big)} <\infty.
%\ee
We assume that a family of (generalized) Fortin operators $Q^\delta \in \cL\big(H^2_0(\Omega),Y^\delta\big)$ exists with 
\be \label{eq:13}
\sup_{\delta \in I} \|Q^\delta\|_{\cL(H^2_0(\Omega),Y^\delta)}<\infty, \,\,\,(B^\delta X^\delta)\big(\ran (\identity-Q^\delta)\big)=0.
\ee

From \eqref{eq:15} and \eqref{eq:13}, for $z \in X^\delta$ we have
\begin{align*}
\|B z\|_{H^{-2}(\Omega)}&=
\sup_{0 \neq v \in H^2_0(\Omega)}\frac{|(B^\delta z)(Q^\delta v)|}{\|\triangle v\|_{L_2(\Omega)}} \leq 
\|Q^\delta\|_{ \cL(H^2_0(\Omega),Y^\delta)} \sup_{0 \neq v \in H^2_0(\Omega)}\frac{|(B^\delta z)(Q^\delta v)|}{\|Q^\delta v\|_{Y^\delta}} 
\\ 
&\lesssim \sup_{0 \neq v \in Y^\delta }\frac{|(B^\delta z)(v)|}{\|v\|_{Y^\delta}}=\|B^\delta z\|_{(Y^\delta)'}
 \quad(\delta \in I).
\end{align*}
For $0 \leq \eps \lesssim 1$, defining on $L_2(\Omega)$ the norms
\begin{align*}
\nrm \bigcdot \nrm_{\eps}&:=\sqrt{\|A\bigcdot\|_{H^{-2}(\Omega)\times L_2(\omega)}^2+\eps^2 \|\bigcdot\|_{L_2(\Omega)}^2},
\\
\nrm \bigcdot \nrm_{\eps,\delta}&:=\sqrt{\|B^\delta \bigcdot\|_{(Y^\delta)'}^2+\|C \bigcdot\|_{L_2(\omega)}^2+\eps^2 \|\bigcdot\|_{L_2(\Omega)}^2},
\end{align*}
we have 
$\nrm\bigcdot\nrm_\eps \lesssim \|\bigcdot\|_{L_2(\Omega)}$, and
\be \label{eq:10}
 \nrm \bigcdot \nrm_{\eps} \lesssim \nrm \bigcdot \nrm_{\eps,\delta} \quad\text{on }X^\delta.
\ee

For $(\ell^\delta,q) \in (Y^\delta)' \times L_2(\omega)$, where $\ell^\delta$ approximates $\ell \in H^{-2}(\Omega)$, we define
\be \label{eq:lssol}
u^\delta_\eps:=\argmin_{z \in X^\delta} \big\{\|B^\delta z-\ell^\delta\|_{(Y^\delta)'}^2+\|C z-q\|_{L_2(\omega)}^2+\eps^2\|z\|_{L_2(\Omega)}^2\big\}.
\ee
It is the first component of the pair $(u_\eps^\delta,v_\eps^\delta) \in X^\delta \times Y^\delta$ that solves the \emph{saddle point problem}
\be \label{eq:saddle2}
\left\{
\begin{array}{@{}c@{}c@{}c@{}cr}
\langle v_\eps^\delta,\tilde{v}\rangle_{Y^\delta} & +  (B^\delta u_\eps^\delta)(\tilde{v}) & =\,\, & \ell^\delta(\tilde{v}) & (\tilde{v}\in Y^\delta),\\
(B^\delta \tilde{z})(v_\eps^\delta) & -\langle C u_\eps^\delta,C \tilde{z}\rangle_{L_2(\omega)}-\eps^2\langle u_\eps^\delta,\tilde{z}\rangle_{L_2(\Omega)} & =\,\, &- \langle q,C \tilde{z}\rangle_{L_2(\omega)} & (\tilde{z} \in X^\delta).
\end{array}
\right.
\ee

First we establish well-posedness of \eqref{eq:saddle2} in an appropriate sense.

\begin{lemma} \label{lem:1} The linear mapping
\begin{align*}
&M^\delta:=X^\delta \times Y^\delta \rightarrow (X^\delta \times Y^\delta)'\colon
(u_\eps^\delta,v_\eps^\delta) \mapsto \\
&\big((\tilde{z},\tilde{v}) \mapsto \langle v_\eps^\delta,\tilde{v}\rangle_{Y^\delta} +  (B^\delta u_\eps^\delta)(\tilde{v}) +(B^\delta \tilde{z})(v_\eps^\delta) -\langle C u_\eps^\delta,C \tilde{z}\rangle_{L_2(\omega)}-\eps^2\langle u_\eps^\delta,\tilde{z}\rangle_{L_2(\Omega)}\big)
\end{align*}
is invertible, and $
\nrm u_\eps^\delta  \nrm_{\eps,\delta}+\|v_\eps^\delta\|_{Y^\delta} \eqsim  \|M^\delta(u_\eps^\delta,v_\eps^\delta)\|_{((X^\delta,\nrm\bigcdot\nrm_{\eps,\delta})\times Y^\delta)'}$.
\end{lemma}

\begin{proof} The uniform boundedness of $M^\delta$, i.e., the inequality `$\gtrsim$', follows from the definition of $\nrm\bigcdot\nrm_{\eps,\delta}$.
To show the other direction, consider an equation $M^\delta(u_\eps^\delta,v_\eps^\delta)=f^\delta \in (X^\delta \times Y^\delta)'$.
We introduce the additional variables 
$\theta_{\eps}^\delta:=-C u_{\eps}^\delta$ and $\mu_{\eps}^\delta:=-\eps u_{\eps}^\delta$, and the bilinear form
$$
e(z;(\tilde{v},\tilde{\theta},\tilde{\mu})):=(B^\delta z)(\tilde{v})+\langle C z, \tilde{\theta}\rangle_{L_2(\omega)}+\eps \langle z,\tilde{\mu}\rangle_{L_2(\Omega)}
$$
on $X^\delta\times \big(Y^\delta\times L_2(\omega) \times  L_2(\Omega)\big)$.
Then, one verifies that $M^\delta(u_\eps^\delta,v_\eps^\delta)=f^\delta$ is equivalent to
finding $(u^\delta_{\eps},v^\delta_{\eps},\theta^\delta_{\eps},\mu^\delta_{\eps}) \in X^\delta \times \big(Y^\delta \times L_2(\omega) \times L_2(\Omega)\big)$ such that
\be \nonumber
\begin{split}
& \langle (v^\delta_{\eps},\theta^\delta_{\eps},\mu^\delta_{\eps}),(\tilde{v},\tilde{\theta},\tilde{\mu})\rangle_{Y^\delta \times L_2(\omega) \times L_2(\Omega)}+e(u_{\eps}^\delta;(\tilde{v},\tilde{\theta},\tilde{\mu}))\\
&\hspace*{6cm}  +
e(\tilde{z};(v_{\eps}^\delta,\theta^\delta_{\eps},\mu^\delta_{\eps}))=f^\delta(\tilde{z},\tilde{v})
\end{split}
\ee
for all $(\tilde{z},\tilde{v},\tilde{\theta},\tilde{\mu}) \in X^\delta \times \big(Y^\delta \times L_2(\omega) \times L_2(\Omega)\big)$.

Given $z \in X^\delta$, let $\tilde{v} \in Y^\delta$ be such that $\|\tilde{v}\|_{Y^\delta}=\|B^\delta z\|_{(Y^\delta)'}$ and $(B^\delta z)(\tilde{v}) = \|B^\delta z\|_{(Y^\delta)'}^2$. Taking $(\tilde{\theta},\tilde{\mu}):=(Cz,\eps z)$, we have
\begin{align*}
 e(z;(\tilde{v},\tilde{\theta},\tilde{\mu}))
&= \|B^\delta z\|_{(Y^\delta)'}^2 +\|C z\|_{L_2(\omega)}^2+\eps^2\|z\|_{L_2(\Omega)}^2
\\
&= \nrm z\nrm_{\eps,\delta}  \sqrt{
\|\tilde{v}\|_{Y^\delta}^2 +\|\tilde{\theta}\|_{L_2(\omega)}^2+\|\tilde{\mu}\|_{L_2(\Omega)}^2}.
\end{align*}
From this LBB stability, we conclude
$$
\nrm u^\delta_{\eps}\nrm_{\eps,\delta}+\|v^\delta_{\eps}\|_{Y^\delta}+\|\theta^\delta_{\eps}\|_{L_2(\omega)}+\|\mu^\delta_{\eps}\|_{L_2(\Omega)}  \lesssim  \|f^\delta\|_{(X^\delta,\nrm \bigcdot\nrm_{\eps,\delta})\times Y^\delta)'}. \qedhere
$$
\end{proof}

The following theorem generalizes Theorem~\ref{thm:ls} to possible nonconforming test spaces.

\begin{theorem} \label{thm:lsnonconf} Let $(\ell^\delta,q) \in (Y^\delta)' \times L_2(\omega)$ and $u \in L_2(\Omega)$, and assume \eqref{eq:12}--\eqref{eq:13}. Then the solution $u_\eps^\delta \in X^\delta$ of \eqref{eq:lssol} satisfies
\be \label{eq:23}
\begin{split}
\|A(u-&u_\eps^\delta)\|_{H^{-2}(\Omega)\times L_2(\omega)}+\eps\|u-u_\eps^\delta\|_{L_2(\Omega)}  \lesssim\\
&\|Au-(\ell,q)\|_{H^{-2}(\Omega) \times L_2(\omega)}+\eps\|u\|_{L_2(\Omega)}+\\
&\min_{z \in X^\delta} \Big\{\| u-z\|_{L_2(\Omega)}+  \sup_{0 \neq v \in Y^\delta}\frac{|(B^\delta z)(v-E^\delta v)+\ell(E^\delta v)-\ell^\delta(v)|}{\|v\|_{Y^\delta}}\Big\}.
\end{split}
\ee
Consequently, if, in addition,
\be \label{eq:22}
\begin{split}
\sup_{0 \neq v \in Y^\delta}&\frac{|(B^\delta z)(v-E^\delta v)+\ell(E^\delta v)-\ell^\delta(v)|}{\|v\|_{Y^\delta}} \lesssim \\
&\hspace*{5em}\|\ell-Bu\|_{H^{-2}(\Omega)}+\| u-z\|_{L_2(\Omega)}+\osc_\delta(\ell),
\end{split}
\ee
where $\osc_\delta(\ell)$ is some `\emph{data-oscillation}' term, then
$$
\|A(u-u_\eps^\delta)\|_{H^{-2}(\Omega)\times L_2(\omega)}\!+\eps\|u-u_\eps^\delta\|_{L_2(\Omega)}  \lesssim
\cE_{\rm data}+\cE_{\rm appr}(\delta)+\osc_\delta(\ell)+\eps\|u\|_{L_2(\Omega)}.
$$
\end{theorem}

\begin{proof}
For arbitrary $z \in X^\delta$, using $\nrm\bigcdot\nrm_\eps \lesssim \|\bigcdot\|_{L_2(\Omega)}$ and \eqref{eq:10}, we estimate
$$
\nrm u-u_\eps^\delta\nrm_\eps \leq \nrm u-z\nrm_\eps +\nrm z-u_\eps^\delta\nrm_\eps \lesssim \| u-z\|_{L_2(\Omega)} +\nrm z-u_\eps^\delta\nrm_{\eps,\delta}.
$$
With $v^\delta_\eps$ being the second component of the solution of \eqref{eq:saddle2}, an application of Lemma~\ref{lem:1} shows that
$$
\nrm z-u_\eps^\delta\nrm_{\eps,\delta} \leq \nrm z-u_\eps^\delta\nrm_{\eps,\delta} +\|-v_\eps^\delta\|_{Y^\delta} \eqsim
\sup_{0 \neq (\tilde{z},\tilde{v}) \in X^\delta \times Y^\delta}
\frac{|M^\delta(z-u_\eps^\delta,-v_\eps^\delta)(\tilde{z},\tilde{v})| }
{\sqrt{\nrm \tilde{z}\nrm_{\eps,\delta}^2+\|\tilde{v}\|^2_{Y^\delta}}}.
$$
We estimate
\begin{align*}
& |
M^\delta(z-u_\eps^\delta,-v_\eps^\delta)(\tilde{z},\tilde{v})
|=|M^\delta(z,0)(\tilde{z},\tilde{v})
-\ell^\delta(\tilde{v})+\langle q,C \tilde{z}\rangle_{L_2(\omega)}
|\\
& \leq |(B^\delta z-\ell^\delta)(\tilde{v})|+
|\langle q,C \tilde{z}\rangle_{L_2(\omega)}-\langle C z,C \tilde{z}\rangle_{L_2(\omega)}-\eps^2 \langle z,\tilde{z}\rangle_{L_2(\Omega)}|.
\end{align*}
We bound the second term in this upper bound as follows
\begin{align*}
&|\langle q-Cz,C \tilde{z}\rangle_{L_2(\omega)}-\eps^2 \langle z,\tilde{z}\rangle_{L_2(\Omega)}| \\
&\leq 
(\|q-Cu\|_{L_2(\omega)}+\|C (u-z)\|_{L_2(\omega)})\|C\tilde{z}\|_{L_2(\omega)}+\\
& \hspace*{15em}(\eps\|u-z\|_{L_2(\Omega)}+\eps\|u\|_{L_2(\Omega)})\eps\|\tilde{z}\|_{L_2(\Omega)}\\
&\lesssim 
\big(\|q-Cu\|_{L_2(\omega)}+ \|u-z\|_{L_2(\Omega)}  +\eps\|u\|_{L_2(\Omega)}\big)\nrm \tilde{z}\nrm_{\eps,\delta}.
\end{align*}
Using \eqref{eq:15} we write the first term as
\be \label{eq:24}
(B^\delta z-\ell^\delta)(\tilde{v})=
(B^\delta z)(\tilde{v}-E^\delta\tilde{v})+\ell(E^\delta\tilde{v})-\ell^\delta(\tilde{v})+
(B z)(E^\delta \tilde{v})-\ell(E^\delta \tilde{v}),
\ee
and, using \eqref{eq:12}, estimate
\begin{align*}
|(B z)(E^\delta \tilde{v})-\ell(E^\delta \tilde{v})| &\lesssim \|B z -\ell\|_{H^{-2}(\Omega)} \| \tilde{v}\|_{Y^\delta}\\
&\lesssim (\|u-z\|_{L_2(\Omega)}+\|B u-\ell\|_{H^{-2}(\Omega)} )\| \tilde{v}\|_{Y^\delta},
\end{align*}
and
\begin{align*}
|(B^\delta z)(\tilde{v}-E^\delta\tilde{v})+&\ell(E^\delta\tilde{v})-\ell^\delta(\tilde{v})|\\
&\leq \sup_{0 \neq v \in Y^\delta}\frac{|(B^\delta z)(v-E^\delta v)+\ell(E^\delta v)-\ell^\delta(v)}{\|v\|_{Y^\delta}} \|\tilde{v}\|_{Y^\delta}.
\end{align*}

By collecting above estimates, and by minimizing over $z \in X^\delta$, the proof of \eqref{eq:23} is completed.
\end{proof}

For completeness we state the following analogue of Corollary~\ref{RScorol:1}.

\begin{corollary} \label{corol:2} In the setting of  Theorem~\ref{thm:lsnonconf},  let $\eps>0$ be such that
$$
\eps \|u\|_{L_2(\Omega)} \eqsim \cE_{\rm data}+\cE_{\rm appr}(\delta)+\osc_\delta(\ell).
$$
Then with $\alpha \in (0,1]$ from Theorem~\ref{RSprop:1}, it holds that
\be \nonumber
\begin{split}
\|u-u_\eps^\delta\|_{L_2(G)}  \lesssim 
\big( \cE_{\rm data}+&\cE_{\rm appr}(\delta)+\osc_\delta(\ell)\big)^\alpha \\
&\times \max\big(\cE_{\rm data}+\cE_{\rm appr}(\delta)+\osc_\delta(\ell),\|u\|_{L_2(\Omega)}\big)^{1-\alpha}.
\end{split}
\ee
\end{corollary}

\section{Piecewise constant trial spaces with Morley test spaces} \label{sec:6}
Let $\tria^\delta$, $\cN^\delta$, $\cN_\circ^\delta$, $\cE^\delta$, and $\cE_\circ^\delta$ be as in Sect.~\ref{sec:4}.
Let $X^\delta$ be the space of \emph{piecewise constants} w.r.t.~$\tria^\delta$. Let $Y^\delta$ be the \emph{Morley} finite element space w.r.t.~$\tria^\delta$, being the space of piecewise quadratic polynomials that are continuous at $\cN_\circ^\delta$, and vanish at $\cN^\delta \setminus \cN_\circ^\delta$, and whose normal derivatives are continuous at midpoints 
of $\cE_\circ^\delta$, and vanish at midpoints of $\cE^\delta \setminus \cE_\circ^\delta$, see Figure~\ref{fig:2}.
\begin{figure}[h]
\includegraphics[width=4cm]{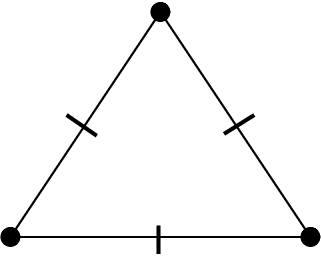}
\caption{\label{fig:2} Morley element and its 6 DoF.}
\end{figure}
We equip $Y^\delta$ with the Hilbertian norm
$$
\|\bigcdot\|_{Y^\delta}:=\sqrt{\sum_{K \in \tria^\delta}|\bigcdot|_{H^2(K)}^2},
$$
which is uniformly equivalent to $\sqrt{\sum_{K \in \tria^\delta}\|\bigcdot\|_{H^2(K)}^2}$ (see \cite[Lem.~8]{310.65}), and define
$$
(B^\delta z)(v):=-\sum_{K \in \tria^\delta} \int_K z \triangle v\,dx,
$$
which satisfies \eqref{eq:15}.

We set the operator $Q^\delta \in \cL(H^2_0(\Omega),Y^\delta)$ by
\be \label{eq:weerFortin}
(Q^\delta v)(\nu)=v(\nu)\quad (\nu \in \cN_\circ^\delta), \quad
\int_e \partial_{n_e} Q^\delta v \,ds=\int_e \partial_{n_e} v\,ds \quad (e \in \cE_\circ^\delta)
\ee
(notice that $\int_e \partial_{n_e} Q^\delta v \,ds=h_e (\partial_{n_e} Q^\delta v)(m_e)$ is single-valued).
It is known, e.g.~see \cite{37.474}, that
$$
\|v -Q^\delta v\|_{H^k(K)} \lesssim h_K^{2-k} |v|_{H^2(K)} \quad (0 \leq k \leq 2,\,v \in H^2(K))
$$
% eigenlijk alleen $\|Q^\delta v\|_{H^2(K)} \lesssim \|v\|_{H^2(K)}$ gebruikt.
For $z \in X^\delta$, $v \in H^2_0(\Omega)$, integration-by-parts shows that
\begin{align*}
(Bz)(v)=-\sum_{K \in \tria^\delta} \int_K z \triangle v\,dx& =
-\sum_{e \in \cE^\delta_\circ} \int_e \llbracket z \rrbracket_e \partial_{n_e} v\,ds=
-\sum_{e \in \cE^\delta_\circ} \int_e \llbracket z \rrbracket_e \partial_{n_e} Q^\delta v\,ds\\
&=
-\sum_{K \in \tria^\delta} \int_K z \triangle Q^\delta v\,dx=(B^\delta z)(Q^\delta v),
\end{align*}
so that $Q^\delta$ is a valid  (generalized) Fortin operator, i.e., it satisfies \eqref{eq:13}.
\medskip

Following \cite{306.52}, we define a map $\widetilde{E}^\delta \in \cL(Y^\delta,H^2_0(\Omega))$ into the HCT finite element space w.r.t.~$\tria^\delta$ from \eqref{eq:HCT} as follows.
Recall that a HCT function is determined by its function values at $\nu \in \cN_\circ^\delta$, its normal derivatives at midpoints $m_e$ of $e \in \cE_\circ^\delta$, and its gradients at $\nu \in \cN_\circ^\delta$. The first two types of DoF are shared by the Morley space $Y^\delta$,
and for $v \in Y^\delta$ we set $(\widetilde{E}^\delta v)(\nu):=v(\nu)$ ($\nu \in \cN_\circ^\delta$), and $(\partial_{n_e} \widetilde{E}^\delta v)(m_e):=(\partial_{n_e} v)(m_e)$ ($e \in \cE_\circ^\delta$). For $\nu \in \cN_\circ^\delta$, we select \emph{some} $\tria^\delta \ni K \ni \nu$, and set
$(\nabla \widetilde{E}^\delta v)(\nu):=(\nabla v|_{K})(\nu)$.

As follows from \cite{306.52}, it holds that 
\be \label{eq:29}
h_K^{-4}\|v-\widetilde{E}^\delta v\|_{L_2(K)}^2+h_K^{-2}|v-\widetilde{E}^\delta v|_{H^1(K)}^2+|\widetilde{E}^\delta v|_{H^2(K)}^2
 \lesssim \sum_{K' \in \omega(K)} |v|_{H^2(K')}^2
 \ee
 ($v \in Y^\delta,\,K \in \tria^\delta$). Actually, in \cite[(proof of) Prop.~3.17]{306.52} only 
$ |v-\widetilde{E}^\delta v|_{H^2(K)}^2
 \lesssim \sum_{K' \in \omega(K)} |v|_{H^2(K')}^2$ was shown, but the remaining parts from \eqref{eq:29} follow from \cite[Lemma~3.15]{306.52}, and obvious bounds for the $\|\bigcdot\|_{L_2(\Omega)}$ and $|\bigcdot|_{H^1(\Omega)}$ (semi-) norms of nodal basis functions of the HCT space that supplement such bounds in $|\bigcdot|_{H^2(\Omega)}$ semi-norm from \cite[Lemma 3.14]{306.52}.

To construct a companion operator $E^\delta$ that has the \emph{additional property} to be a right-inverse of the Fortin interpolator, we modify $\widetilde{E}^\delta$. As in Sect.~\ref{sec:4}, for $e \in \cE_\circ^\delta$ let $b_e$ be the nodal basis function from the HTC space that corresponds to the DoF associated to the midpoint $m_e$ of $e$.
For $v \in Y^\delta$, we set $E^\delta \in \cL(Y^\delta,H^2_0(\Omega)$ by
$$
E^\delta v:=\widetilde{E}^\delta v+\sum_{e \in \cE_\circ^\delta} \frac{\int_e \partial_{n_e}(\identity-\widetilde{E}^\delta)v\,ds}{\int_e \partial_{n_e} b_e\,ds} b_e.
$$
From $|b_e|_{H^k(\Omega)} \eqsim h_e^{2-k}$ ($0 \leq k \leq 2$), $\int_e \partial_{n_e} b_e \,ds=\frac23 h_e$,
$$
|\int_e \partial_{n_e}(\identity-E^\delta)v\,ds| \lesssim
\|(\identity-\widetilde{E}^\delta)v\|_{H^1(K)}+h_e|(\identity-\widetilde{E}^\delta)v|_{H^2(K)},
$$
 where $\tria^\delta \ni K \supset e$, and \eqref{eq:29}, we infer that
\be \label{eq:17}
h_K^{-4}\|v-E^\delta v\|_{L_2(K)}^2+h_K^{-2}|v-E^\delta v|_{H^1(K)}^2+|E^\delta v|_{H^2(K)}^2
 \lesssim \sum_{K' \in \omega(K)} |v|_{H^2(K')}^2
\ee
  ($v \in Y^\delta,\,K \in \tria^\delta$), so that in particular $E^\delta$ satisfies \eqref{eq:12}.
  
 Moreover, by construction it holds that
  $$
  \int_e \partial_{n_e} E^\delta v\, ds= \int_e \partial_{n_e} v\, ds \quad(e \in \cE^\delta,\,v \in Y^\delta),
  $$
  and because $E^\delta$ reproduces function values at the nodal points, from \eqref{eq:weerFortin} we conclude that indeed
 \be \label{eq:28}
  Q^\delta E^\delta=\identity.
 \ee
 
 Finally, notice that the values of DoF associated to $\nu \in \cN^\delta_\circ$ of $E^\delta v$ are equal to those of $\widetilde{E}^\delta v$, which will be relevant in Sect.~\ref{sec:7}.
\begin{remark}
 Other than in e.g.~\cite{75.265, 306.52,37.475}, we have \emph{not} extended the HCT space with additional higher order polynomial bubble functions to ensure \eqref{eq:28}.
 \end{remark}
 
 We have established proporties \eqref{eq:12}--\eqref{eq:13}. For the application of Theorem~\ref{thm:lsnonconf} and Corollary~\ref{corol:2} it remains to specify $\ell^\delta \in (Y^\delta)'$ ($\delta \in \Delta$) which approximate $\ell \in Y'$, and to give a corresponding definition of $\osc_\delta(\ell)$ that verifies \eqref{eq:22}. We consider \emph{two options} for $(\ell^\delta)_{\delta \in I}$. As in \cite{35.68,306.51,37.476} for elliptic problems, with the first option the nonconforming test functions are smoothed before submitted to $\ell \in H^{-2}(\Omega)$.
 
 \begin{theorem}  \label{thm:1} Let $(\ell,q) \in H^{-2}(\Omega) \times L_2(\omega)$, and let $(X^\delta,Y^\delta)_{\delta \in I}$ be as specified above.
Define $\ell^\delta \in (Y^\delta)'$ by
$$
\ell^\delta(v):=\ell(E^\delta v).
$$
Then \eqref{eq:22} is valid with $\osc_\delta(\bigcdot)=0$.
 \end{theorem}
 
  \begin{proof} Thanks to \eqref{eq:13} and \eqref{eq:28},  $(z,v) \in X^\delta\times Y^\delta$ it holds that
$$
 (B^\delta z)(v-E^\delta v)=(B^\delta z)(v)-(B^\delta z)(Q^\delta E^\delta v)=0,
$$
 and $\ell(E^\delta v)-\ell^\delta(v)=0$. 
  \end{proof}
  
 With above definition of $\ell^\delta$, the bounds from Theorem~\ref{thm:lsnonconf} and Corollary~\ref{corol:2}  are \emph{qualitatively equal} to those from 
 Theorem~\ref{thm:ls} and Corollary~\ref{RScorol:1}  for conforming test functions.
 A disadvantage, however, is that, although not in the system matrix, the `complicated' HTC finite element space does enter the evaluation of the forcing term.

 At the price of restricting to $\ell \in L_2(\Omega)$, and getting a non-zero oscillation term in the upper bound (which is, however, of higher order than $\min_{z \in X^\delta} \|u-z\|_{L_2(\Omega)}$), this disadvantage disappears by not smoothing the nonconforming test functions, being our second option:

 \begin{theorem} \label{thm:theorempje} Let $(\ell,q) \in L_2(\Omega) \times L_2(\omega)$, $u \in L_2(\Omega)$, and let $(X^\delta,Y^\delta)_{\delta \in I}$ be as specified above. Take
 $$
 \ell^\delta:=\ell.
 $$
 For arbitrary, but fixed $r \in \N$, let
$$
\osc_\delta(\ell):=\sqrt{\sum_{T \in \tria^\delta}\min_{\widehat{\ell} \in \cP_r(K)} h_K^4\|\ell-\widehat{\ell}\|_{L_2(K)}^2}. 
$$
Then
$$
\sup_{0 \neq v \in Y^\delta}\frac{|(B^\delta z)(v-E^\delta v)+\ell(E^\delta v-v)|}{\|v\|_{Y^\delta}} \lesssim
\|\ell-Bu\|_{H^{-2}(\Omega)}+\| u-z\|_{L_2(\Omega)}+\osc_\delta(\ell).
$$
\end{theorem}

\begin{proof} Because for $(z,v) \in X^\delta\times Y^\delta$, we have $(B^\delta z)(v-E^\delta v)=0$, it remains to verify
$|\ell(E^\delta v-v)| \lesssim \big(\|\ell-Bu\|_{H^{-2}(\Omega)}+\|u-z\|_{L_2(\Omega)}+\osc_\delta(\ell)\big)\|v\|_{Y^\delta}$.

For $v \in Y^\delta$, using \eqref{eq:17} we have
\be \label{eq:18}
\begin{split}
&|\ell(E^\delta v-v)| =|\sum_{K \in \tria^\delta} \int_K \ell (E^\delta v-v)\,dx|\\
& \leq \sum_{K \in \tria^\delta} h_K^2\|\ell\|_{L_2(K)} h_K^{-2} \|E^\delta v-v\|_{L_2(K)}  \lesssim \sqrt{\sum_{K \in \tria^\delta} h_K^4\|\ell\|^2_{L_2(K)}} \|v\|_{Y^\delta}.
\end{split}
\ee

Following \cite{77.4}, to proceed we use the bubble function technique (\cite{307}).
For $K \in \tria^\delta$ and $\widehat{\ell} \in \cP_r(K)$, there exists a `bubble' $\phi=\phi_K \in H^2_0(K)$ with $\|\phi \|_{L_2(K)} \eqsim \|\widehat{\ell}\|_{L_2(K)}$, $ \|\widehat{\ell}\|_{L_2(K)}^2 \lesssim \int_K \widehat{\ell} \phi\,dx$, and $\|\phi\|_{L_2(K)} \eqsim h_K^2 |\phi|_{H^2(K)}$.
For $z \in X^\delta$, using that $\int_K z \triangle \phi \,dx=0$ we estimate
\begin{align*}
\|\widehat{\ell}\|_{L_2(K)}^2 &\lesssim \int_K \widehat{\ell} \phi\,dx= \int_K (\widehat{\ell} -\ell)\phi\,dx+ \int_K \ell \phi\,dx\\
&= \int_K (\widehat{\ell} -\ell)\phi\,dx+\int_K(z-u) \triangle \phi\,dx+\int_K \ell \phi+u \triangle \phi\,dx\\
&\lesssim \big(\|\widehat{\ell} -\ell\|_{L_2(K)}+h_K^{-2} \|u-z\|_{L_2(K)} + h_K^{-2}\|\ell-Bu\|_{H^{-2}(K)}\big)\|\widehat{\ell}\|_{L_2(K)},
\end{align*}
and so
\begin{align*}
\|\ell\|_{L_2(K)}&\leq \|\ell-\widehat{\ell}\|_{L_2(K)}+\|\widehat{\ell}\|_{L_2(K)} \\
&\leq \|\widehat{\ell} -\ell\|_{L_2(K)}+h_K^{-2} \|u-z\|_{L_2(K)} + h_K^{-2}\|\ell-Bu\|_{H^{-2}(K)}.
\end{align*}
or
\be \label{eq:20}
\sum_{K \in \tria^\delta} h_K^4\|\ell\|^2_{L_2(K)} \lesssim \osc_\delta(\ell)^2+\|u-z\|_{L_2(\Omega)}^2+\|\ell-Bu\|_{H^{-2}(\Omega)}^2,
\ee
where we have used that $\sum_{K \in \tria^\delta} \|\bigcdot\|_{H^{-2}(K)}^2 \lesssim \|\bigcdot\|_{H^{-2}(\Omega)}^2$.
Indeed, given $f \in H^{-2}(\Omega)$, take $\phi_K \in H^2_0(K)$ with $\|f\|_{H^{-2}(K)} \eqsim \frac{f(\phi_K)}{\|\phi_K\|_{H^2(\Omega)}}$ and $\|\phi_K\|_{H^2(K)}=\|f\|_{H^{-2}(K)}$. Then $\sum_{K \in \tria^\delta}\|f\|_{H^{-2}(K)}^2 \eqsim \frac{(\sum_K f(\phi_K))^2}{\sum_K \|\phi_K\|_{H^2(K)}^2}=\frac{f(\sum_K \phi_K)^2}{\|\sum_K \phi_K\|_{H^2(K)}^2} \lesssim \|f\|_{H^{-2}(\Omega)}^2$.
By combining \eqref{eq:18} and \eqref{eq:20}, the proof is completed.
\end{proof}

\begin{remark}
For the analysis of the use of nonconforming test spaces, in Theorems~\ref{thm:lsnonconf} and \ref{thm:theorempje} we applied the `\emph{medius analysis}' introduced by T.~Gudi in \cite{77.4} (for Galerkin discretizations of elliptic problems). 

Assuming that $\triangle u \in L_2(\Omega)$, alternatively in \eqref{eq:24} for $z \in X^\delta$ and $\tilde{v} \in Y^\delta$
one may expand $
(B^\delta z -\ell^\delta)(\tilde{v})$ in \eqref{eq:24} as
$$
B^\delta (z-u)(\tilde{v})-(\ell^\delta+\triangle u)(\tilde{v})+ \sum_{K \in \tria^\delta} \int_K \tilde{v} \triangle u-u \triangle \tilde{v}\,dx.
$$
The first two terms in modulus are bounded by a multiple of $(\|u-z\|_{L_2(\Omega)}+\|\ell^\delta+\triangle u\|_{(V^\delta)'})\|\tilde{v}\|_{V^\delta}$.
For $Y^\delta$ being the Morley finite element space w.r.t.~$\tria^\delta$, and assuming that $u \in H^1(\Omega)$, 
the remaining consistency error known from \emph{Strang's second lemma} can be bounded by 
\begin{align*}
 |\sum_{K \in \tria^\delta} \int_K\tilde{v}\triangle u -u \triangle \tilde{v}\,dx|&=\big|\sum_{e \in \cE^\delta} \int_e \partial_{n_e}u \llbracket \tilde{v}\rrbracket-u \llbracket \partial_{n_e}\tilde{v}\rrbracket\,ds\big|\\
&\lesssim \sqrt{\sum_{K\in \tria^\delta} h_K^2|u|_{H^1(K)}^2+h_K^4\|\triangle u\|_{L_2(K)}^2}\, \|\tilde{v}\|_{V^\delta}
\end{align*}
(e.g.~consult \cite[pp.~1383--1384]{20.12}). Taking $\ell^\delta:=\ell$, this leads to the bound
\be \label{eq:26}
\begin{split}
\|A(u-&u_\eps^\delta)\|_{H^{-2}(\Omega)\times L_2(\omega)}+\eps\|u-u_\eps^\delta\|_{L_2(\Omega)}  \lesssim\\
&\|\ell+\triangle u\|_{(V^\delta)'}+\|Cu -q\|_{L_2(\omega)}+\eps\|u\|_{L_2(\Omega)}+\min_{z \in X^\delta} \| u-z\|_{L_2(\Omega)}\\
&+\sqrt{\sum_{K\in \tria^\delta} h_K^2|u|_{H^1(K)}^2+h_K^4\|\triangle u\|_{L_2(K)}^2}.
\end{split}
\ee

Unlike Theorem~\ref{thm:theorempje}, this bound requires a priori smoothness conditions on $u$, whereas moreover the oscillation term in Theorem~\ref{thm:theorempje} is of higher order than  the square root term in \eqref{eq:26}.
\end{remark}

\begin{remark}
We expect that in any case the result of Theorem~\ref{thm:theorempje} extends to polytopes $\Omega \subset \R^n$ for $n \geq 3$.
The definition of the Morley finite element space $Y^\delta$ has been extended to $n \geq 3$ in \cite{310.65}.
The construction of the Fortin projector $Q^\delta$ from \eqref{eq:weerFortin} also extends to $n \geq 3$.
Some companion operator $E^\delta \in \cL(Y^\delta,H^2_0(\Omega))$ that satisfies \eqref{eq:17}, and so \eqref{eq:12}, can also be constructed for $n \geq 3$.
For the result of Theorem~\ref{thm:theorempje} the additional property \eqref{eq:28} was convenient but not necessary.
Without \eqref{eq:28}, the proof of Theorem~\ref{thm:theorempje} is completed by 
$$
|(B^\delta z)(v-E^\delta v)| \lesssim
\sqrt{\sum_{e \in \cE_\circ^\delta} h_e \|\llbracket z \rrbracket_e\|_{L_2(e)}^2} \|v\|_{Y^\delta},\quad(v \in Y^\delta,\,z \in X^\delta),
$$
and, using the bubble function technique, by
$$
\sum_{e \in \cE_\circ^\delta} h_e \|\llbracket z \rrbracket_e\|_{L_2(e)}^2
\lesssim \|u-z\|_{L_2(\Omega)}^2+\|\ell-Bu\|_{H^{-2}(\Omega)}^2+\osc_\delta(\ell)^2. \qedhere
$$
\end{remark}

\section{Least squares approximations for the Cauchy problem} \label{sec:7}
In Sections~\ref{sec:3}--\ref{sec:6}, a regularized least squares method was presented for solving the UC problem in ultra-weak variational form $Au=(Bu,Cu)=(\ell,q)$ (see Theorem~\ref{RSprop:2}). Here we briefly discuss modifications needed for the application of this least squares method for solving the Cauchy problem in the ultra-weak variational form $Bu=f$, where $(Bz)(v):=-\int_\Omega z \triangle v\,dx$, and $f(v):=\ell(v)+\int_\Sigma \psi v-g \partial_n v\,ds$ ($z \in L_2(\Omega)$, $v \in H^2_{0,\Sigma^c}(\Omega)$) (see Theorem~\ref{thm:condstabCauchy}) such that the key results Theorem~\ref{thm:ls}--Corollary~\ref{RScorol:1} and Theorem~\ref{thm:lsnonconf}--Corollary~\ref{corol:2}  remain valid.
Since the case $|\Sigma^c|=0$ is not relevant, in the following we assume that $|\Sigma^c|>0$.
Further, we assume that $\overline{\Sigma}$, and so $\overline{\Sigma^c}$, are unions of edges from the triangulations $\tria^\delta$.

The first adaptation needed is to omit from the saddle-point formulations \eqref{eq:saddle1} (conforming test spaces) or  \eqref{eq:saddle2} (nonconforming test spaces) terms that contain the operator $C$. Furthermore the functional $\ell$ should be replaced by $f$.

Although the operator $B$ has the same definition, for the UC problem its acts on $v \in H^2_0(\Omega)$, and for the Cauchy problem on $v \in H^2_{0,\Sigma^c}(\Omega)$. On $H_0^2(\Omega)$ it holds that $\|\bigcdot\|_{H^2(\Omega)} \eqsim \|\triangle \bigcdot\|_{L_2(\Omega)}$, which is not true on $H^2_{0,\Sigma^c}(\Omega)$. What does hold on the latter space is that $\|\bigcdot\|_{H^2(\Omega)} \eqsim |\bigcdot|_{H^2(\Omega)}$.
So for the Cauchy problem, the scalar product $\langle \triangle \bigcdot, \triangle \bigcdot\rangle_{L_2(\Omega)}$ in \eqref{eq:saddle1} has to be replaced by $\sum_{|\alpha|=2}\langle \partial^\alpha \bigcdot, \partial^\alpha\bigcdot\rangle_{L_2(\Omega)}$.

For the UC problem, as test spaces we considered HTC and Morley finite element spaces where DoF associated to points on $\partial\Omega$ were set to zero. For the Cauchy problem only DoF associated to points on $\overline{\Sigma^c}$ has to be set to zero. With this adaptation the constructions of valid Fortin \eqref{fortin}, or generalized Fortin operators \eqref{eq:13} directly extend to the Cauchy problem.
Inspection of the proof from \cite{310.65} reveals that, thanks to $|\Sigma^c|>0$, also in setting of the Cauchy problem $\sum_{K \in \tria^\delta} |\bigcdot|_{H^2(K)}^2 \eqsim \sum_{K \in \tria^\delta} \|\bigcdot\|_{H^2(K)}^2$ is valid on the Morley finite element space, so that the scalar product in \eqref{eq:saddle2} does not have to be adapted.
The construction of the companion $E^\delta$ with the properties \eqref{eq:17} and  \eqref{eq:28} directly extends to the setting of the Cauchy problem.

For $Y^\delta$ being the Morley finite element space, what remains is a definition of $f^\delta \in (Y^\delta)'$ and that of $\osc_\delta(f)$ that verifies \eqref{eq:22}. As in Theorem~\ref{thm:1} a first option is to take $f^\delta(\bigcdot)=f(E^\delta \bigcdot)$ in which case \eqref{eq:22} is valid with $\osc_\delta(\bigcdot)=0$. As noticed before, a disadvantage of this option is that the `complicated' HTC space enters the evaluation of the forcing term.

Similarly to Theorem~\ref{thm:theorempje}, a second option is to restrict to $\ell \in L_2(\Omega)$, and, say, to $\psi, g \in L_2(\Sigma)$, and to extend 
that theorem to this setting. For non-zero $g$, however, we did not manage to establish a bound of type \eqref{eq:22} for $\osc_\delta(f)$ being a true oscillation-term, i.e., an expression that vanishes for $f$, $\psi$ and $g$ from spaces of piecewise polynomials w.r.t.~$\tria^\delta$ or $\{e\colon e \in \cE^\delta,\,e \subset \Sigma\}$, respectively.

Therefore, for $\ell \in L_2(\Omega)$, as a third option we propose to take
\be \label{eq:third}
f^\delta(v)=\int_\Omega \ell v \,dx+\int_\Sigma \psi E^\delta v-g \partial_n E^\delta v\,ds.
\ee
Then $f(E^\delta v)-f^\delta(v)=\int_\Omega \ell(E^\delta v-v)\,dx$, and with $\osc_\delta(\ell)$ from Theorem~\ref{thm:theorempje}, by following its proof we find
\begin{align*}
\sup_{0 \neq v \in Y^\delta}\!\frac{|(B^\delta z)(v\!-\!E^\delta v)\!+\!f(E^\delta v)\!-\!f^\delta(v)|}{\|v\|_{Y^\delta}} & \lesssim
\|\ell\!-\!Bu\|_{H^{-2}(\Omega)}\!+\!\| u\!-\!z\|_{L_2(\Omega)}\!+\!\osc_\delta(\ell)\\
&\!\!\!\!\!\!\!\!\leq \|f\!-\!Bu\|_{(H_{0,\Sigma^c}(\Omega))'}\!+\!\| u\!-\!z\|_{L_2(\Omega)}\!+\!\osc_\delta(\ell).
\end{align*}

To evaluate \eqref{eq:third}, for $e \in \{\widetilde{e}\colon \widetilde{e}\in \cE^\delta,\,\widetilde{e} \subset \Sigma\}$, and $v$ from the Morley space $Y^\delta$, we need $(E^\delta v)|_e \in P_3(e)$ and $(\partial_{n_e} E^\delta v)|_e \in P_2(e)$.
The function $(E^\delta v)|_e$ is determined by its values at both $\nu \in \cN^\delta$ on $e$, which are the values of $v$ at these nodes, together with the values of its first order derivatives at these $\nu$, which for each of these $\nu$ is given by the first order derivative at $\nu$ of $(v|_{K_\nu})|_e$ for some $\tria^\delta \ni K_\nu \ni \nu$.
The function $(\partial_{n_e} E^\delta v)|_e$ is determined by its values at both $\nu \in \cN^\delta$ on $e$, which for each of these $\nu$
 is given by the value of $(\partial_{n_e} v|_{K_\nu})(\nu)$, and the value of $\int_e \partial_{n_e} E^\delta v \,ds$ which is equal to $\int_e \partial_{n_e} v \,ds=h_e (\partial_{n_e} v)(m_e)$.
 
 We conclude that this third option for $f^\delta$ does not require the construction of the HTC basis functions, whereas, for $\ell \in L_2(\Omega)$, it results in an oscillation term which is of higher order than the error of best approximation $\inf_{z \in X^\delta} \|u-z\|_{L_2(\Omega)}$.

\section{Numerical experiments} \label{sec:8}

\subsection{Unique continuation}
On the square domain $\Omega = (-1,1)^2$, and subdomain $\omega = (-\tfrac12,\tfrac12)^2$, for data $\ell$ and $q$ we consider the unique continuation problem of finding $u\in L_2(\Omega)$ that satisfies, in distributional sense, 

$$
 \left\{
\begin{array}{r@{}c@{}ll}
-\Delta u&\,\,=\,\,& \ell &\text{ on } \Omega,\\
u &\,\,=\,\,& q &\text{ on } \omega.
\end{array}
\right.
$$

Let $\bbT$ denote the collection of all conforming triangulations that can be created by newest vertex bisections starting from the initial triangulation that consists of 4 triangles created by cutting $\Omega$ along its diagonals. The interior vertex of the initial triangulation of both squares are labelled as the `newest vertex' of all four triangles in both squares.

For $\tria^\delta \in \bbT$, we choose $X^\delta$ as the space of \emph{piecewise constants} w.r.t.~$\tria^\delta$, and $Y^\delta$ as the \emph{Morley} finite element space w.r.t.~$\tria^\delta$ with all DoFs associated to boundary points set to zero. 
In the experiments for this unique continuation problem, we consider uniform triangulations $\tria^\delta$ only.

Following Section~\ref{sec:5}--\ref{sec:6}, the numerical approximation $u^\delta_\eps$ of $u$ is computed as the first component of the pair $(u^\delta_\eps, v^\delta_\eps)\in X^\delta\times Y^\delta$ that satisfies
\begin{align*}
&\sum_{K \in \tria^\delta}\int_K \big\{\sum_{|\alpha|=2} \partial^\alpha v_\eps^\delta \partial^\alpha \tilde{v}\big\} - u_\eps^\delta \triangle \tilde{v} - \tilde{z} \triangle v^\delta_\eps \, dx- \eps^2\langle u^\delta_\eps, \tilde{z}\rangle_{L_2(\Omega)} - \langle u^\delta_\eps, \tilde{z}\rangle_{L_2(\omega)} \\
&\hspace*{10em} = \ell^\delta(\tilde{v}) - \langle q, \tilde{z}\rangle_{L_2(\omega)}
 ,\qquad ((\tilde{z}, \tilde{v} ) \in X^\delta\times Y^\delta).
\end{align*}

We will consider the prescribed solutions $u(x,y) = e^x\cos(y)$ and $u(x,y) = \tfrac{1}{4 \pi} \log(x^2 + y^2)$, and corresponding \emph{exact} data
$(\ell,q)=(-\triangle u,u|_\omega)$. Because of the latter, and the fact that the best approximation error in $L_2(\Omega)$ behaves at best as $\sim \rm{DoFs}^{-\frac12}$, we have set $\eps:=\rm{DoFs}^{-\frac12}$. Setting $\eps=0$, however, gave similar results.

In the first, smooth case, we have $\ell=0$, and we took $\ell^\delta=0$.
In the second case, we have $\ell = \delta_0$, where $\delta_0\colon v \mapsto v(0)$ is the Dirac functional.
Because $\delta_0 \not\in (Y^\delta)'$, we set $\ell^\delta := \ell \circ E^\delta$, where $E^\delta$ is the companion operator defined in Section~\ref{sec:6}.

For different choices of $\omega \subseteq G\subseteq \Omega$, Figure \ref{figure:UC} shows the $L_2(G)$-error in the finite element approximations $u^\delta_\eps$.
For the smooth solution $u(x,y) = e^x\cos(y)$, the convergence rates vary between the optimal rate $\frac12$ for $G=\omega$, being in correspondence with Footnote~\ref{voetje}, and the rate $\frac{1}{10}$ for $G=\Omega$, actually a case which is not covered by the theory.

The results for $u(x,y) = \tfrac{1}{4 \pi} \log(x^2 + y^2)$ are similar, although the rate $\frac12$ for $G=\omega$ is not exactly attained, which can be understood by noticing that this $u$ is just not in $H^1(\Omega)$.

\begin{figure}[h]
%\hspace*{-1cm}   
\centering
\begin{subfigure}{0.5\textwidth}
\includegraphics[width = \textwidth]{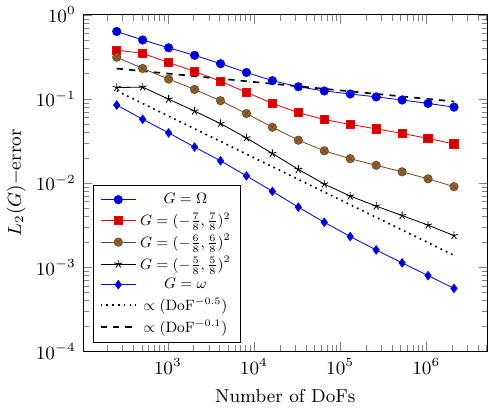}
%\label{figure:UC_smooth}
\end{subfigure}%
\begin{subfigure}{0.5\textwidth}
\includegraphics[width = \textwidth]{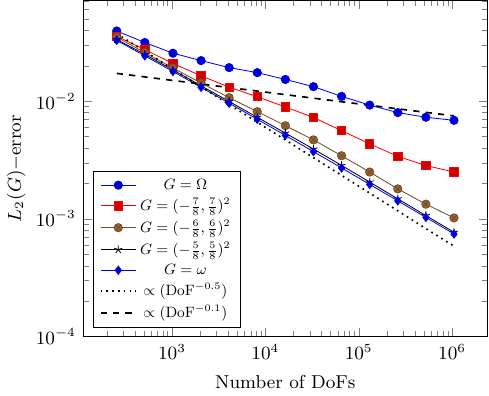}
%\label{figure:UC_log}
\end{subfigure}%
\caption{DoFs in $X^\delta$ vs.~error in $u_\eps^\delta$ in $L_2(G)$-norm for different choices of $\omega \subseteq G\subseteq \Omega$. Left: error for $u(x,y) = e^x\cos(y)$.  Right: error for $u(x,y) = \tfrac{1}{4 \pi} \log(x^2 + y^2)$. In both cases $\eps = \mbox{DoFs}^{-\frac12}$.}
\label{figure:UC}
\end{figure}

\subsection{Cauchy problem}
On the rectangular domain $\Omega = (-1,1)\times (0,1)$, and Cauchy boundary $\Sigma=(-1,1)\times\{0\}$, for data $\ell$, $g$, and $\psi$ we consider the Cauchy problem of finding $u \in L_2(\Omega)$ that satisfies, in distributional sense, 
$$
 \left\{
\begin{array}{r@{}c@{}ll}
-\Delta u&\,\,=\,\,& \ell &\text{ on } \Omega,\\
u &\,\,=\,\,& g &\text{ on } \Sigma,\\
\nabla u \cdot \vec{n}&\,\,=\,\,& \psi &\text{ on } \Sigma.
\end{array}
\right.
$$
We prescribe the harmonic solution $u(x,y) := \tilde u(x,y) - \tilde u(-x,y)$ (see Figure~\ref{fig:u}) where $\tilde u(r,\theta) := r^{-\frac12} \sin (-\frac{\theta}{2})$ in polar coordinates, and determine the exact data correspondingly.
One may verify that $u \in H^{\frac12-\eps}(\Omega)$ for all $\eps>0$, whereas $u \not\in H^{\frac12}(\Omega)$.

\begin{figure}[h]
    \centering
    \includegraphics[width = 0.5\textwidth]{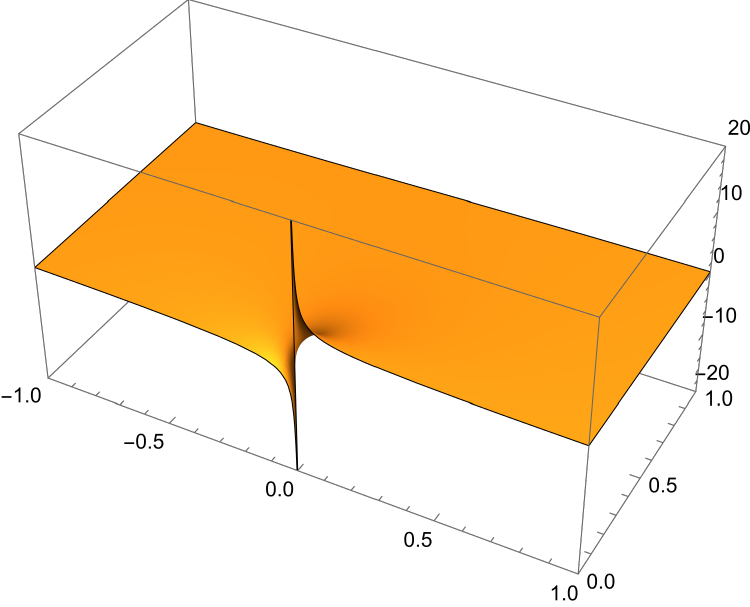} \hspace*{-0.8em}
        \raisebox{4ex}{\includegraphics[width = 0.5\textwidth]{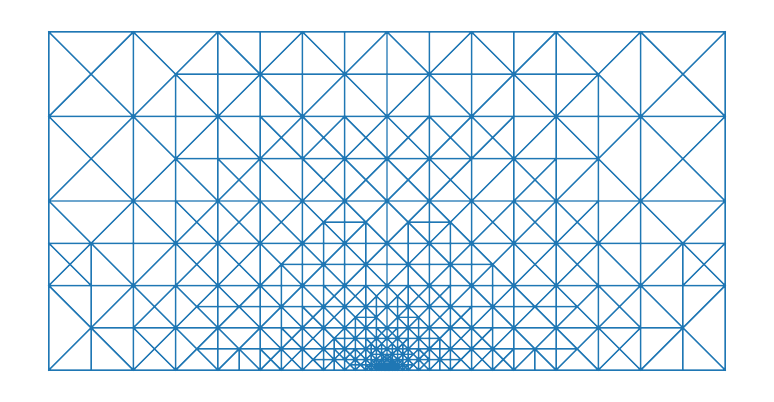}}
    \caption{Left: Plot of $(x,y)\mapsto \min(\max(u(x,y),20),-20)$. Right: Example of a mesh generated by our adaptive routine.}
    \label{fig:u}
\end{figure}

Let $\bbT$ denote the collection of all conforming triangulations that can be created by newest vertex bisections starting from the initial triangulation that consists of 8 triangles created by first cutting $\Omega$ along the y-axis into two equal squares, and then cutting these squares along their diagonals. The interior vertex of the initial triangulation of both squares are labelled as the `newest vertex' of all four triangles in both squares.

For any $\tria^\delta \in \bbT$, we choose $X^\delta$ as the space of \emph{piecewise constants} w.r.t.~$\tria^\delta$, and $Y^\delta$ as the \emph{Morley} finite element space w.r.t.~$\tria^\delta$ with all DoFs associated to points on $\overline{\Sigma^c}$ set to zero.

Following Section ~\ref{sec:7} the numerical approximation $u^\delta_\eps$ of $u$ is computed as the first component of the pair $(u^\delta_\eps, v^\delta_\eps)\in X^\delta\times Y^\delta$ that satisfies

\begin{align*}
&\sum_{K \in \tria^\delta}
\int_K \big\{\sum_{|\alpha|=2} \partial^\alpha v_\eps^\delta \partial^\alpha \tilde{v}\big\} - u_\eps^\delta \triangle \tilde{v} - \tilde{z} \triangle v^\delta_\eps \, dx
- \eps^2\langle u^\delta_\eps, \tilde{z}\rangle_{L_2(\Omega)} - \langle u^\delta_\eps, \tilde{z}\rangle_{L_2(\omega)} \\
&\hspace*{10em} = \int_\Sigma \psi E^\delta \tilde{v}-g \partial_n E^\delta \tilde{v}\,ds
 ,\qquad ((\tilde{z}, \tilde{v} ) \in X^\delta\times Y^\delta).
\end{align*}

We will measure the $L_2(G)$-error for $G = (-\frac12,\frac12)\times(0,\frac12)$. Notice that $G$ 'touches' the singularity of $u$ in $(0,0)$. 

Because of the latter, and since only $u \in H^{\frac12-\eps}(\Omega)$ for $\eps>0$, even if $u_\eps^\delta$ would be a quasi-best $L_2(G)$-approximation to $u$ then for uniform meshes the convergence rate could not exceed $\frac14$.
We therefore investigate an adaptive refinement strategy. Following the idea of the Zienkiewicz \& Zhu estimator (\cite{320.65}), we define a continuous piecewise linear $\widehat{u}^\delta$ w.r.t.~$\tria^\delta$ by 
$$
\widehat{u}^\delta(\nu) = \frac{\sum_{\{K\in \mathcal{T}^\delta \colon K \ni \nu\} }|K| u^\delta_\eps|_K}{\{\sum_{K\in \tria^\delta\colon K\ni \nu\}}|K|},
\qquad (\nu \in \{\text{vertices of }\tria^\delta\}),
$$
and then approximate $\| u - u_\eps^\delta\|^2_{L_2(\Omega)}$ by $\sum_{K\in \tria^\delta}\eta_K$, where $\eta_K := \|\widehat{u}^\delta - u_\eps^\delta\|^2_{L_2(K)}$. Using these local error indicators, we apply a bulk chasing strategy with parameter $\theta=0.6$.

\subsubsection{Exact data}
Besides for our singular solution $u$ depicted in Figure~\ref{fig:u}, first we computed our approximation $u_\eps^\delta$ for the smooth, harmonic solution 
$u(x,y) := \sqrt{\tfrac{2}{\pi}}\sin(x)\sinh(y)$ using exact data $g$ and $\psi$.
The results shown in the left picture of Figure~\ref{figure:adaptive} for uniform meshes and $\eps = \mbox{DoFs}^{-\frac12}$ show the optimal convergence rate $\frac12$ despite the ill-posedness of the Cauchy problem. 

\begin{figure}
\centering
\begin{subfigure}{0.5\textwidth}
\includegraphics[width = \textwidth]{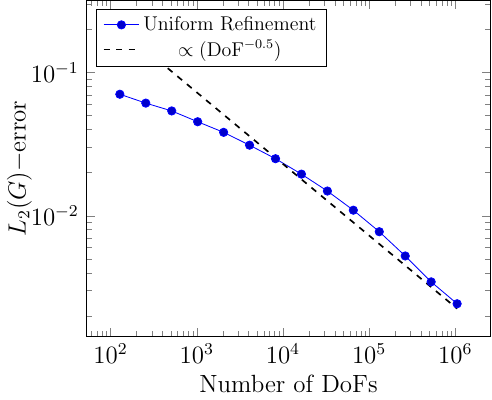}
%\caption{}
%\label{figure:adaptive_b}
\end{subfigure}%
\begin{subfigure}{0.5\textwidth}
\includegraphics[width = \textwidth]{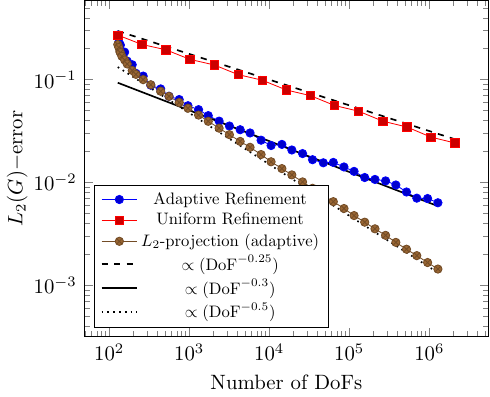}
%\caption{}
%\label{figure:adaptive_a}
\end{subfigure}%

\caption{DoFs in $X^\delta$ vs.~error in $L_2(G)$, exact data. Left: error for smooth solution with $\eps = \mbox{DoFs}^{-\frac12}$.
Right: error for singular solution $u$, as well as $\|Q^\delta u - u\|_{L_2(\Omega)}$ where $Q^\delta$ is the $L_2$-projection onto $X^\delta$ corresponding to the adaptive refined $\tria^\delta$. Here $\eps = \mbox{DoFs}^{-\frac12}$ or $\eps=0$ in the adaptive or uniform refinement case. }
\label{figure:adaptive}
\end{figure}

For the singular solution $u$ and exact data $g$ and $\psi$ the results for uniform meshes and $\eps=0$, and adaptive meshes and $\eps = \mbox{DoFs}^{-\frac12}$ are given in the right picture of Figure~\ref{figure:adaptive}.
These results show that adaptive refinement gives a quantitative improvement, although the asymptotic rate improves only slightly from $0.25$ (being optimal for uniform refinements) to $0.3$.
To verify whether our adaptive strategy produces quasi-optimal meshes we also computed the $L_2(\Omega)$-orthogonal projection of the exact solution onto the corresponding piecewise constant finite element space, and indeed the resulting errors decay with the optimal rate $\frac12$.

For the numerical experiments presented in the next two subsubsections,  the singular solution $u$  depicted in Figure~\ref{fig:u} is considered, and starting from the initial mesh all meshes $\tria^\delta$ have been created  using the adaptive strategy described above.

\subsubsection{Randomly perturbed data}
We demonstrate the effect of adding random perturbations to the Dirichlet datum.
Noticing that $\|v\mapsto\int_\Sigma g \partial_n v\,ds\|_{H^2_{0,\Sigma^c}(\Omega)'} \eqsim \|g\|_{H^{-\frac12}(\Sigma)}$, 
given $\tria^\delta$ we chose a random piecewise constant $g^\delta$ w.r.t.~$\cE^\delta \cap \overline{\Sigma}$ with $\mathcal{E}_{\rm data} \eqsim \|g^\delta\|_{H^{-\frac12}(\Sigma)} \approx 0.05$. 
We achieved this by normalizing a random function in $\cE^\delta \cap \overline{\Sigma}$, taking values in $[0,1]$, in a discrete $H^{-\frac12}(\Sigma)$-norm that on  piecewise constant functions is uniformly equivalent to the true $H^{-\frac12}(\Sigma)$-norm, and then multiplied the result with $0.05$. We used the discrete $H^{-\frac12}(\Sigma)$-norm constructed in \cite[p211]{75.258} using results from \cite{13.6}. 

Replacing the exact Dirichlet datum $g$ by $g+g^\delta$, in Figure \ref{figure:random} we compare the regularization strategies $\eps = \mathcal{E}_{\rm data} + \mbox{DoFs}^{-\frac12}$ and $\eps = \mbox{DoFs}^{-\frac12}$. The results indicate that such random perturbations are rather harmless, and therefore do not require additional regularization.

\begin{figure}
%\hspace*{-1cm}   
\centering
\begin{subfigure}{0.5\textwidth}
\includegraphics[width = \textwidth]{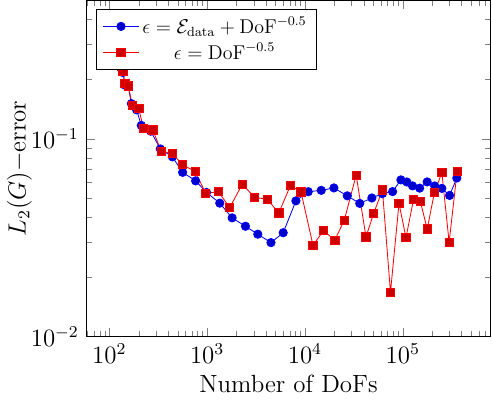}
\caption{}
\label{figure:random_fig}
\end{subfigure}%
\caption{DoFs in $X^\delta$ vs.~error in $L_2(G)$ in the case of random piecewise contant perturbations w.r.t.~current (boundary) mesh of the Dirichlet datum with $H^{-\frac12}(\Sigma)$-norm $\approx 0.05$.}
\label{figure:random}
\end{figure}

\subsubsection{'Difficult' perturbations}
We will construct `tough' perturbations in the following way.
One can show that for $m\in \mathbb{N}$, the solution $u = u^{(m)}$ of the Cauchy problem with data $(\ell,g,\psi)= (0,g^{(m)}, 0)$ where $g^{(m)}(x):=\sqrt{\tfrac{m \pi}{2}}\sin( \tfrac{m\pi}{2}(x + 1))$ is given by $u^{(m)}(x,y) = \sqrt{\tfrac{m \pi}{2}}\sin( \tfrac{m\pi}{2}(x + 1))\cosh(\tfrac{m\pi}{2}y)$. It holds that $\|g^{(m)}\|_{H^{-\frac12}(\Sigma)} = 1$, whereas $\|u^{(m)}\|_{L_2(\Omega)}\sim \tfrac{1}{2\sqrt{2}}e^{\tfrac{m\pi}{2}}$ ($m \rightarrow \infty$) illustrating the ill-posedness of the Cauchy problem.

Replacing the exact Dirichlet datum $g$ by $g+0.05 g^{(m)}$, we again compare the regularization strategies of choosing  $\eps = \mathcal{E}_{\rm data} + \mbox{DoFs}^{-\frac12}$ and $\eps = \mbox{DoFs}^{-\frac12}$. For intermediate values of $m$, Figure \ref{figure:difficultperturbation} shows that additional regularization because of the data error is needed. An explanation why for larger $m$ such regularization is not helpful is that the resulting errors are not yet resolved on the meshes that were employed.

\begin{figure}
%\hspace*{-1cm}   
\centering
\begin{subfigure}{0.5\textwidth}
\includegraphics[width = \textwidth]{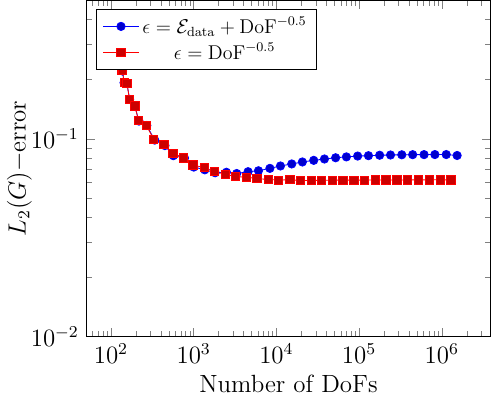}
\caption{$m=1$}
\label{figure:m1}
\end{subfigure}%
\begin{subfigure}{0.5\textwidth}
\includegraphics[width = \textwidth]{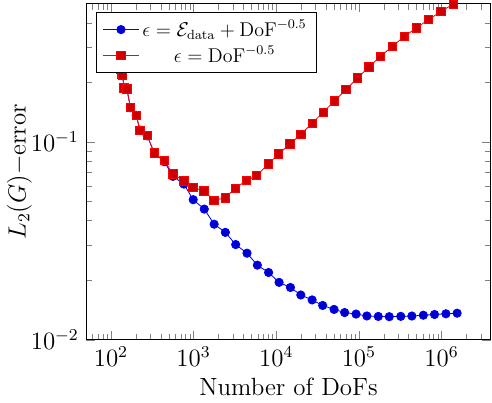}
\caption{$m=4$}
\label{figure:m4}
\end{subfigure}
\begin{subfigure}{0.5\textwidth}
\includegraphics[width = \textwidth]{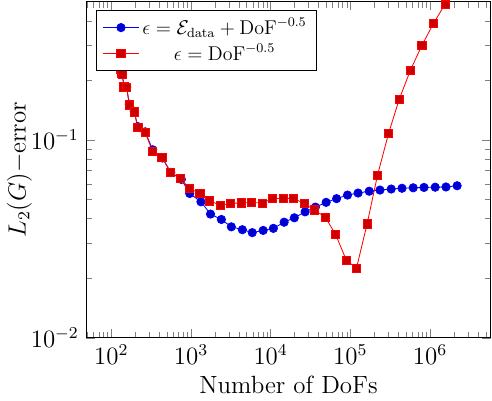}
\caption{$m=6$}
\label{figure:m6}
\end{subfigure}%
\begin{subfigure}{0.5\textwidth}
\includegraphics[width = \textwidth]{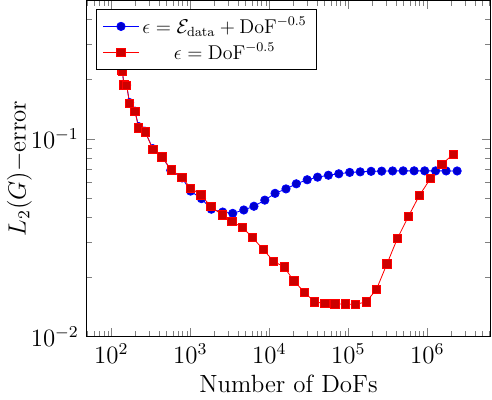}
\caption{$m=16$}
\label{figure:m16}
\end{subfigure}%

\caption{DoFs in $X^\delta$ vs.~error in $L_2(G)$ for different choices of $\eps$ and perturbation of the Dirichlet data with $0.05*g_m$.}
\label{figure:difficultperturbation}
\end{figure}

\section{Conclusion} \label{sec:9} \enlargethispage*{2em}
We have established conditional stability estimates for the Unique Continuation and Cauchy problems in ultra-weak variational formulations, which estimates are the basis of regularized least squares formulations.
Replacing dual norms in the least squares functional by computable discretized dual norms leads to a mixed formulation.
For uniformly stable pairs of trial- and test spaces, the $L_2$-error on a subdomain can be bounded by the best possible fractional power of the sum of the data error and the error of best approximation. In comparison to standard variational formulations they show a qualitatively best possible achievable error decay at lower computational cost. The use of $C^1$ finite element test functions can be avoided by the application of nonconforming test functions. \pagebreak

%\bibliographystyle{halpha}
%\bibliography{../../ref}
\end{document}